\documentclass [12pt] {article}

\usepackage{amsfonts}    
\usepackage{amsmath}
\usepackage{amsthm}
\usepackage{amssymb} 
\usepackage{stackengine}
\usepackage{graphicx}
\usepackage{marvosym}
\usepackage{wasysym}
\usepackage{ dsfont }
\usepackage{tikz-cd} 
\usepackage{verbatim}
\usepackage{float}
\theoremstyle{definition}
\newtheorem{defn}{Definition}[subsection]

\newtheorem{ex}{Example}[section]
\newtheorem{thm}{Theorem}[subsection]
\newtheorem{prop}{Proposition}[subsection]
\newtheorem{cor}{Corollary}[subsection]
\newtheorem{lemma}{Lemma}[subsection]
\newtheorem{rmk}{Remark}[subsection]

\newcommand{\qq}{\mathbb{Q}}

\newcommand{\zz}{\mathbb{Z}}
\newcommand{\nn}{\mathbb{N}}

\usepackage[labelsep=endash]{caption}

\newcommand\stacklongleftrightarrow[1]{%
    \mathrel{{\stackon[4pt]{$\longleftrightarrow$}{$\scriptscriptstyle#1$}}}}

\usepackage{color}

\usepackage[top=.75in, bottom=.75in, left=1in, right=1in]{geometry}
\title{Mixed Dimer Configuration Model in Type $D$ Cluster Algebras}
\author{Gregg Musiker and Kayla Wright}
\date{January 23, 2023}

\begin{document}

\maketitle

\begin{abstract}
We define a combinatorial model for $F$-polynomials and $g$-vectors for type $D_n$ cluster algebras where the associated quiver is acyclic.  Our model utilizes a combination of dimer configurations and double dimer configurations which we refer to as mixed dimer configurations. In particular, we give a graph theoretic recipe that describes which monomials appear in such $F$-polynomials, as well as a graph theoretic way to determine the coefficients of each of these monomials. In addition, we prove that a weighting on our mixed dimer configuration model yields the associated $g$-vector. To prove this formula, we use a combinatorial formula due to Thao Tran \cite{tran} and  provide explicit bijections between her combinatorial model and our own.
\end{abstract}

\begin{section}{Introduction}

Positivity of the coefficients in the Laurent expansions of cluster variables was a long-running open question in the theory of cluster algebras, until it was proven around 2014 \cite{LeeSchiffler, ghkk}.  Even with positivity verified, explicit combinatorial interpretations for the Laurent expansions as generating functions has still been a question of active research for many different cluster algebras.  In the case of cluster algebras of classical type, i.e. those defined by a (valued) quiver mutation-equivalent to a Dynkin Diagram of type $A_n$, $B_n$, $C_n$, or $D_n$, there are various such interpretations in the literature.  For example, see \cite{musiker-classical, fomin2003systems, tran, yz} depending on whether the initial quiver is bipartite, acyclic, or more general.  The Caldero-Chapoton formula, as first appearing in \cite{caldero2006cluster} for ADE types, also yields positive formulas, modulo verification that the Euler characteristics of the relevant quiver Grassmannians are positive.
\allowbreak  \vspace{1em}

Furthermore, when a cluster algebra comes from a surface, there is a combinatorial interpretation of cluster variables via the language of snake graphs \cite{mswpos}.  The only cluster algebras that are both of classical type and from a surface are those of type $A_n$ or type $D_n$.  Such cluster algebras are associated to triangulations of $(n+3)$-gons or once-punctured $n$-gons, respectively.\allowbreak  \vspace{1em}

Going beyond cluster algebras of classical type or cluster algebras from surfaces, there are families of cluster algebras associated to coordinate rings of the Grassmannians \cite{scott2006grassmannians}, or more generally positroid varieties \cite{postnikov2006total, ssbw}.  For such cluster algebras, a finite subset of cluster variables (namely those parametrized by Pl\"ucker coordinates), admit combinatorial interpretations using (almost) perfect matchings of plabic graphs.  Such interpretations are often packaged together into what Postnikov called the boundary measurement map.  See \cite{lam2015dimers, marsh2016twists, muller2017twist}.  However, beginning with examples such as cluster algebras associated to the Grassmannian of $3$-planes in $6$-space, additional cluster variables lack such an interpretation.\allowbreak  \vspace{1em}

Another rich source of cluster algebras, which contain (toric) cluster variables with combinatorial interpretations also given via dimers on bipartite graphs, are those associated to the octahedron recurrence \cite{speyeroct}, Gale-Robinson sequences \cite{speyeroct,BMPW}, T-systems \cite{DiF14}, or brane tilings \cite{EF, FHK06, HV}.  The first author has in particular studied such cases in \cite{jmz, LMNT, LaiMus1, LaiMus2}.  We also note the work of \cite{Glick-Pentagram} which focused on related cluster algebras, but explicitly focused on $Y$-system and $F$-polynomial interpretations rather than expansions of cluster variables.\allowbreak  \vspace{1em}

More recent work \cite{kenpem}, also see \cite{LaiMus2}, has developed (and proven or conjectured) a double dimer configuration combinatorial interpretation for cluster variables whose Laurent expansions were previously not understood combinatorially. 
Related work appears in \cite[Sections 3-4]{lam2015dimers}, where products of exactly two (resp. three) Pl\"ucker coordinates are interpreted as Temperley-Lieb Immanants (resp. webs). See also the work of \cite{helen} that shows that partition functions of double dimers on certain graphs with nodes satisfy three-term identities. With these examples in mind, we have returned to the classical case of type $D_n$ cluster algebras arising from acyclic quivers with an eye towards a combinatorial interpretation for Laurent expansions of cluster variables that utilized a mixture of dimer configurations and double dimer configurations.  To this end, we demonstrate a combinatorial interpretation via mixed dimer configurations for $F$-polynomials in the case of acyclic quivers of type $D_n$.  Thanks to the separation of addition formulae of Fomin and Zelevinsky \cite{ca4}, the $F$-polynomials and $g$-vectors provide as much information as the Laurent expansions of the cluster variables (with principal coefficients) themselves.  We anticipate that this new interpretation will provide further insight into the conjectured mixed dimer configuration combinatorial interpretations arising in \cite{LaiMus2}, or allow the generalization of techniques using weighted decompositions of Quiver Grassmannian as in \cite{glickweyman}.\allowbreak  \vspace{1em}

We begin this paper by reviewing the basics of cluster algebras in Section \ref{sec:background}. We then describe our combinatorial model, culminating with our main theorem for $F$-polynomials in Section \ref{sec:model}.  The proof of this theorem (Theorem \ref{thm:main}) is given in Section \ref{sec:proof} by introducing Thao Tran's model from \cite{tran} and then showing the two directions of our combinatorial bijection.  In Section \ref{sec:$g$-vectors}, we complete the picture by providing an approach for reading off the associated $g$-vector from the minimal matching $M_-$ associated to $Q$ and $\underline{d}$ as part of our model.  This allows for the computation of the Laurent expansion for the associated cluster variables as well.  We end with a discussion of future directions in Section \ref{sec:future}.

\vspace{1em}

{\bf Acknowledgements:} The authors would like to thank the support of the NSF, grants
DMS-1745638 and DMS-1854162. We would also like to thank Esther Banaian, Sebastian Franco, Helen Jenne, Rick Kenyon, David Speyer, and Lauren Williams for many helpful conversations.  We additionally thank the anonymous referee whose suggestions improved our exposition.  Lastly, we would like to acknowledge the REU report from University of Minnesota's 2018 Summer REU for inspiration and insightful examples \cite{reu}.

\end{section}

\begin{section}{Background}
\label{sec:background}
\begin{subsection}{Cluster Algebra Basics}
Cluster algebras were originally defined by Fomin and Zelevinsky in \cite{fzfoundations} in order to give a combinatorial characterization of total positivity and canonical bases in algebraic groups, but have since been found to be helpful in many other areas of mathematics such as algebraic combinatorics, symplectic geometry, Teichm\"uller theory, and even some areas of physics.  Loosely speaking, a \textit{cluster algebra} is a certain commutative algebra generated by a distinguished set of generators called \textit{cluster variables}. To obtain these cluster variables, we start with an \textit{initial seed}, which consists of $n$ cluster variables as well as a quiver or an exchange matrix. Using the quiver/exchange matrix, we \textit{mutate} in the $k^\text{th}$ direction, transforming our original set of cluster variables into a new set where we change the $k^\text{th}$ cluster variable for a new expression in the rest of the variables based on the \textit{binomial exchange relation}. After mutation, we obtain a new seed containing a new set of cluster variables and we iterate mutation in all possible directions. The total set of resulting cluster variables generates the cluster algebra. \allowbreak \vspace{1em}

In \cite{ca4}, the notion of $F$-polynomials and $g$-vectors were introduced which gave a formula for cluster variables in terms of data only dependent on the initial seed. So, once the $F$-polynomial and associated $g$-vectors are computed, we can recover the Laurent polynomial expansion for the associated cluster variable. 

\begin{subsubsection}{Defining a Cluster Algebra}
In this section, we formally define a cluster algebra of geometric type following \cite{fwz}. 

\begin{defn} 
Let $J$ be a finite set and let $\text{Trop}(u_j~:~j \in J)$ be an abelian group with respect to multiplication freely generated by the elements $u_j$ for $j \in J$. Define \textbf{tropical plus}, denoted $\oplus$, via 

$$\prod_j u_j^{a_j} \oplus \prod_j u_j^{b_j} := \prod_j u_j^{\text{min}(a_j,b_j)}.$$
Call $\text{Trop}(u_j:j \in J)$ a \textbf{tropical semifield} with respect to multiplication and $\oplus$.
\end{defn}

\begin{rmk}
If $J = \emptyset$, then we obtain the trivial semifield consisting of the single element 1. The group ring of $\text{Trop}(u_j:j \in J)$ is the ring of Laurent polynomials in the variables $u_j$.
\end{rmk}

Let $m,n \in \zz_{>0}$ such that $m \geq n$ and let $\mathbb{P} = \text{Trop}(x_{n+1}, \dots, x_m)$. Let $\mathcal{F}$ be the field of fractions in $n$ algebraically independent variables with coefficients in $\qq\mathbb{P}$, the field of fractions of the group ring $\zz\mathbb{P}$. 

\begin{defn} 
A \textbf{labeled seed} in $\mathcal{F}$ is a pair $(\Tilde{x}, \Tilde{B})$ where
\begin{itemize}
    \item $\Tilde{x} = (x_1, \dots, x_n, x_{n+1}, \dots, x_m)$ is the \textbf{extended cluster} where $x_1, \dots, x_n$ are algebraically independent over $\qq\mathbb{P}$ and $\mathcal{F} = \qq \mathbb{P}(x_1, \dots, x_n)$ and 
    \item $\Tilde{B}$ is the \textbf{extended exchange matrix} i.e. an $m \times n$ matrix over $\zz$ such that the top $n \times n$ submatrix $B$ is skew-symmetrizable. That is, the top $n \times n$ submatrix $B$ can be written as a skew symmetric matrix of the form $DB$ for some $n \times n$ diagonal matrix $D$ with positive integer diagonal entries. $B$ is called the \textbf{principal part} of the exchange matrix. 
\end{itemize}
\end{defn}

\begin{rmk}
If we restrict our attention to a skew-symmetric cluster algebra of geometric type, we can replace the notion of an exchange matrix with a quiver.
\end{rmk}

\begin{defn} 
Let $1 \leq k \leq n$. An $m \times n$ matrix $\Tilde{B}'$ is obtained by \textbf{matrix mutation in the $k^{\text{th}}$ direction} from $\Tilde{B}$ if the entries of $\Tilde{B}'=(b_{ij}')$ are given by 
$$b_{ij}' = \begin{cases}
-b_{ij} &\text{if } i=k \text{ or } j=k\\
b_{ij} + \text{sgn}(b_{ik})\max(b_{ik}b_{kj},0) &\text{otherwise}
\end{cases}$$
\end{defn}
\begin{defn} 
Let $(\Tilde{x},\Tilde{B} = (b_{ij}))$ be a labeled seed in $\mathcal{F}$. \textbf{Seed mutation $\mu_k$ in the $k^{\text{th}}$ direction} transforms $(\Tilde{x},\Tilde{B})$ into the labeled seed $\mu_k(\Tilde{x},\Tilde{B}) = (\Tilde{x}',\Tilde{B}')$ where

\begin{itemize}
    \item $\Tilde{x}' = (x_1, \dots, x_k', \dots, x_n)$ where 
    $$x_k' = x_k^{-1}\left(\prod_{i=1}^m x_i^{\max(b_{ik},0)} + \prod_{i=1}^m x_i^{\max(-b_{ik},0)}\right)$$
    \item $\Tilde{B}'$ is obtained by matrix mutation in the direction $k$.
\end{itemize}
\end{defn}

\begin{rmk}
In the skew-symmetric case, there is a notion of quiver mutation that is completely graph theoretic.
\end{rmk}

\begin{defn} 
Let $\mathbb{T}_n$ be the $n$-regular tree. A \textbf{cluster pattern} is an assignment of labeled seed $(\Tilde{x}_t, \Tilde{B}_t)$ to every vertex $t \in \mathbb{T}_n$ such that if $t \xrightarrow{k} t'$, then the seeds labeled by $t$ and $t'$ are obtained from one another via mutation in the $k^{\text{th}}$ direction. Write $\Tilde{x}_t = (x_{1;t}, \dots, x_{m;t})$ and $\Tilde{B}_t = (b_{ij}^t)$.
\end{defn}

\begin{defn} 
Given a cluster pattern $\{(\Tilde{x}_t, \Tilde{B}_t)\}_{t \in \mathbb{T}_n}$, we let $\mathcal{X}$ denote the set of all \textbf{cluster variables}, i.e. 

$$\mathcal{X} := \{x_{j;t}~:~1\leq j \leq n, t \in \mathbb{T}_n\}.$$
The \textbf{cluster algebra} associated to a given cluster pattern is the $\zz\mathbb{P}$ subalgebra of $\mathcal{F}$ generated by all of the cluster variables, i.e. we set $\mathcal{A} = \zz\mathbb{P}[\mathcal{X}]$.
\end{defn}

One of the first fundamental properties proven about cluster algebras is the {\bf Laurent Phenomenon.}

\begin{thm} {\cite[Theorem 3.1]{fzfoundations}} \label{thm:LP}
For $j \in [n]$ and $t \in \mathbb{T}_n$, any cluster variable $x_{j;t} \in \mathcal{A}$ can be expressed as 
$$x_{j;t} = \frac{N(x_1, \dots, x_n)}{x_1^{d_1} \cdots x_n^{d_n}}$$ 
where $N(x_1, \dots, x_n) \in \zz[x_{n+1}^{\pm 1}, \dots, x_m^{\pm 1}][x_1, \dots, x_n]$ and is not divisible by any $x_i$. The \textbf{denominator vector} of $x_{j;t}$ is the vector $\underline{d} = (d_1, \dots, d_n)$ and the denominator vector $\underline{d}$ only depends on $B^0$, the principal part of the exchange matrix, rather than on $\Tilde{B}^0$, the extended exchange matrix.
\end{thm}

It took another decade-and-a-half, but a related fundamental theorem is {\bf Laurent Positivity}.

\begin{thm} {\cite[Theorem 1.1]{LeeSchiffler}/\cite[Theorem 4.10]{ghkk}} \label{thm:pos}
Under the same hypotheses as Theorem \ref{thm:LP}, the polynomial $N(x_1,\dots, x_n)$, appearing in the numerator for the Laurent expansion of $x_{j;t}$, is not only in $\zz[x_{n+1}^{\pm 1}, \dots, x_m^{\pm 1}][x_1, \dots, x_n]$, but is in fact in \allowbreak
$\zz_{\geq 0}[x_{n+1}^{\pm 1}, \dots, x_m^{\pm 1}][x_1, \dots, x_n]$.
\end{thm}

\end{subsubsection} 

\begin{subsubsection}{$F$-Polynomials and $g$-Vectors}
Fix an $n \times n$ skew-symmetrizable matrix $B^0 = (b_{ij}^0)$ over $\zz$ and an initial vertex $t_0 \in \mathbb{T}_n$. Note that a choice of an initial labeled seed $(\Tilde{x}_{t_0}, \Tilde{B}^0)$ determines a cluster pattern.  Here, $\Tilde{B}^0$ is an extended exchange matrix which has $B^0$ as its principal part, and $\Tilde{x}_{t_0}$ is an extended cluster of the appropriate size.

\begin{defn} 
A cluster algebra $\mathcal{A}$ has \textbf{principal coefficients at $t_0$} if $\Tilde{x}_{t_0} = (x_1, \dots, x_{2n})$ and the exchange matrix at $t_0$ is the \textbf{principal matrix} corresponding to $B^0$ given by 
$$\Tilde{B}^{t_0} = \begin{pmatrix}
B^0\\
I
\end{pmatrix}$$
where $I$ is the $n \times n$ identity matrix. Denote this cluster algebra by $\mathcal{A}_{\bullet} = \mathcal{A}_{\bullet}(B^0,t_0)$. Using the initial cluster of $\mathcal{A}_{\bullet}$, define the new variables for $1 \leq j \leq n$

$$\hat{y_j} = x_{n+j} \prod_{i=1}^n x_i^{b_{ij}^0}.$$
\end{defn}

Using these new variables, we are ready to define the $F$-polynomial and $g$-vectors. 
\begin{defn} 
Let $1\leq \ell \leq n$, $t \in \mathbb{T}_n$. There exists a unique primitive polynomial 
$$F_{\ell;t} = F_{\ell;t}^{B^0;t_0} \in \zz[u_1, \dots, u_n]$$
and a unique vector $g_{\ell;t} = g_{\ell;t}^{B^0;t_0} = (g_1, \dots, g_n) \in \zz^n$ such that the cluster variable $x_{\ell;t} \in \mathcal{A}_{\bullet}(B^0,t_0)$ is given by 

$$x_{\ell;t} = F_{\ell;t}^{B^0;t_0}(\hat{y_1}, \dots, \hat{y_n})x_1^{g_1}\cdots x_n^{g_n}.$$
The polynomial $F_{\ell;t}$ is called an \textbf{$F$-polynomial} and $g_{\ell;t}$ is called a \textbf{$g$-vector}.
\end{defn}

The importance of $F$-polynomials and $g$-vectors is motivated by the following fundamental {\bf ``separation of additions''} property of cluster algebras of geometric type.

\begin{thm} {\cite[Theorem 3.7/Corollary 6.3]{ca4}} \label{thm:sep_add}  Given a cluster algebra $\mathcal{A}$ defined over an 
arbitrary semifield $\mathbb{P}$ with an initial labeled seed $(\Tilde{x}_{t_0}, \Tilde{B}^0)$, then all 
cluster variables in $\mathcal{A}$ can be expressed as
$$x_{i;t} = \frac{F_{i;t}^{B^0;t_0}(\hat{y_1}, \dots, \hat{y_n})x_1^{g_1}\cdots x_n^{g_n}}{F_{i;t}^{B^0,t_0}|_\mathbb{P}(y_1, \dots y_n)}.$$
In the denominator, the expression is evaluated in $\mathbb{P}$ using ordinary multiplication and $\oplus$.
\end{thm}

In particular, computing the patterns of $F$-polynomials $\{F_{i; t}\}_{i \in [n];~ t\in \mathbb{T}_n}$ and $g$-vectors $\{g_{i; t}\}_{i \in [n];~ t\in \mathbb{T}_n}$ is sufficient for deriving the cluster pattern $\{\Tilde{x}_t\}_{t \in \mathbb{T}_n}$ regardless of the choice of initial coefficients, i.e. the choice of initial extended exchange matrix $\Tilde{B}^0$. \allowbreak \vspace{1em}

Because of Theorems \ref{thm:pos} and \ref{thm:sep_add}, the $F$-polynomials arising from a cluster algebra have positive integer coefficients.  This observation motivates our desire to have a direct combinatorial explanation for this positivity and integrality, which motivates the combinatorial interpretation studied in this paper.
\end{subsubsection}
\end{subsection}

\begin{subsection}{Cluster Algebras and Finite Type Classification}
In our combinatorial model, we define a graph theoretic interpretation of the $F$-polynomial in ``type $D$." In particular, one of the most striking results about cluster algebras is that the classification of \textit{finite type} cluster algebras is parallel to that of the Cartan-Killing classification of complex simple Lie algebras, or equivalently of crystallographic finite root systems. Such Lie algebras are modeled by finite type Dynkin diagrams, which are built out of irreducible diagrams coming from four infinite families or from five exceptional cases: $A_n, B_n, C_n, D_n, E_6, E_7, E_8, F_4, G_2$.

\begin{defn}
We say that a cluster algebra is of \textbf{finite type} it is has finitely many seeds.
\end{defn}

\begin{defn}
Let $B$ be an $n \times n$ exchange matrix. The \textbf{diagram of $B$}, denoted $\Gamma(B)$, is the weighted directed graph on nodes $v_1, \dots, v_n$ with $v_i$ directed towards $v_j$ if and only if $b_{ij} > 0$ and weighted by $|b_{ij}b_{ji}|$.
\end{defn}

\begin{thm} \cite{ca2}
The cluster algebra $\mathcal{A}$ is of finite type if and only if it has a seed $(x,B)$ such that $\Gamma(B)$ is an orientation of a finite type Dynkin diagram. 
\end{thm}

If the conditions of the above theorem hold, $\mathcal{A}$ is of finite type $X$ where $X$ is one of $A_n, B_n, C_n, D_n, E_6, E_7, E_8, F_4, G_2$, as determined by the Dynkin diagram.

\begin{subsubsection}{$F$-polynomials and Positive Roots}
In order to show the connection of $F$-polynomials to types of a cluster algebra, we have the following setup. Suppose that $B^0$ is an acyclic $n \times n$ exchange matrix of type $A_n, B_n, C_n,$ or $D_n$ i.e. $b_{ij} = \pm a_{ij}$ where $(a_{ij})$ is the Cartan matrix of the corresponding type. Let $\mathcal{A}(B^0,t_0)$ be any cluster algebra which has the extended cluster $(x_1, \dots, x_n, x_{n+1},\dots, x_m)$.  By the Laurent phenomenon, Theorem \ref{thm:LP}, we can associate a denominator vector $\underline{d}$ to each cluster variable $x_{i;t}$ in a given cluster algebra.  For acyclic finite type cluster algebras, such denominator vectors can be described concisely as follows.

\begin{thm} {\cite[Theorem 1.9]{ca2}}
Suppose $\mathcal{A}(B^0,t_0)$ is defined as above, i.e. $B^0 = [b_{ij}]$, where each $b_{ij} = \pm a_{ij}$ and $(a_{ij})$ is the Cartan matrix of a finite root system $\Phi$.  Equivalently, the diagram (resp. quiver) defining $\mathcal{A}(B^0,t_0)$ is an orientation of a (simply-laced) finite type Dynkin diagram. Let $D$ denote the set of all denominator vectors of cluster variables in $\mathcal{A}_{\bullet}(B^0,t_0)$ which do not occur in the initial cluster. Let 
$S = \{\alpha_1, \dots, \alpha_n\}$ be the set of simple roots of the corresponding root system $\Phi$ and let $\Phi_+$ be the set of positive roots of $\Phi$. There is a bijective correspondence between these cluster variables in $\mathcal{A}$ and the positive roots in $\Phi_+$. More specifically, if a cluster variable corresponds to $\alpha = \sum d_i \alpha_i$, then the denominator vector of the cluster variable is $(d_1, \dots, d_n) \in \zz_{\geq 0}^n$.
\end{thm}

Recall that we are working with type $D_n$ cluster algebras, so we focus on the associated root system of type $D_n$. Namely, let $\Phi_+$ be the set of positive roots given by
$$\Phi_+ = \{e_i + \dots + e_j~:~0 \leq i \leq j \leq n-1\} \cup \{e_i + \dots + e_{n-3} + e_{n-1}~:~0 \leq i \leq n-3\} \cup$$
$$\cup \{e_i + \dots + e_{j-1} + 2e_j + \dots + 2e_{n-3} + e_{n-2} + e_{n-1}~:~ 0 \leq i < j \leq n-3\}$$

The data of a positive root $\underline{d} = (d_0, \dots, d_{n-1}) \in \Phi_+$ will be included in our definitions for our graph theoretic interpretation of the $F$-polynomial. 
\end{subsubsection}
\end{subsection}
\end{section}

\begin{section}{Our Model} \label{sec:model}
In this section, we describe our combinatorial model for how to obtain the $F$-polynomial associated to a type $D_n$ cluster algebra where the associated quiver is acyclic. Our model assigns a graph comprised of one hexagon and $n-1$ squares to a given quiver where we attach these tiles based on the orientation of each arrow in the quiver. Relying on a bijection between denominator vectors of cluster variables and positive roots in a type $D_n$ root system, we use the data of the quiver $Q$ and a positive root $\underline{d}$ to assign a mixed dimer configuration to this graph. From this assignment, we create a poset of mixed dimer configurations each of which corresponds to a monomial of the $F$-polynomial associated to $\underline{d}$ and whose coefficient is given by the number of cycles appearing in the mixed dimer configuration. 

\begin{subsection}{Type $A$ Case: Snake Graphs}

\label{sec:snake_A}

We model our approach off of the snake graphs for type $A_n$ cluster algebras.  This can be viewed as a special case of positive formulas for cluster variables associated to arcs in (unpunctured) surfaces \cite{ms, mswpos}.  We also note \cite[Section 2-4]{markoff-comb} describing unpublished work of Carroll-Price during an REU advised by Jim Propp.\allowbreak  \vspace{1em}

In particular, for an acyclic type $A_n$ quiver, i.e. an orientation of the line segment on $n$ vertices, we can build a bipartite planar graph by adjoining square tiles together.  Inductively, an arrow $i \to (i+1)$ or $i \leftarrow (i+1)$ corresponds to adding a tile to the current snake graph to the North or East of the preceding tile, as described by the following rule: when two consecutive arrows are in the same direction, i.e. $(i-1) \to i \to (i+1)$ or $(i-1) \leftarrow i \leftarrow (i+1)$, the corresponding tiles are placed so that they zig-zag (i.e. North then East, or East then North).  On the other hand, when two consecutive arrows are in opposite directions, the corresponding tiles are placed straight (i.e. North then North, or East then East).  The first step of the snake graph (i.e. the placement of the second tile) may be taken arbitrarily although we conventionally take this as a step to the East, without loss of generality.
\allowbreak  \vspace{1em}

Furthermore, we can color the vertices black and white so that for the edge at the border of tiles $i$ and $(i+1)$, if we draw the arrow $i\to (i+1)$ or $i \leftarrow (i+1)$ perpendicular to this edge, then the white vertex appears on the right.  \allowbreak  \vspace{1em}

\begin{ex}\label{ex:snake}
Consider the $A_5$ quiver given by

$$Q~~ = ~~ 0 \longrightarrow 1 \longleftarrow 2 \longleftarrow 3 \longrightarrow 4$$

By the rules described above, this orientation of the $A_5$ quiver yields the following snake graph on 5 square tiles.

\begin{center}
    \includegraphics[scale=.2]{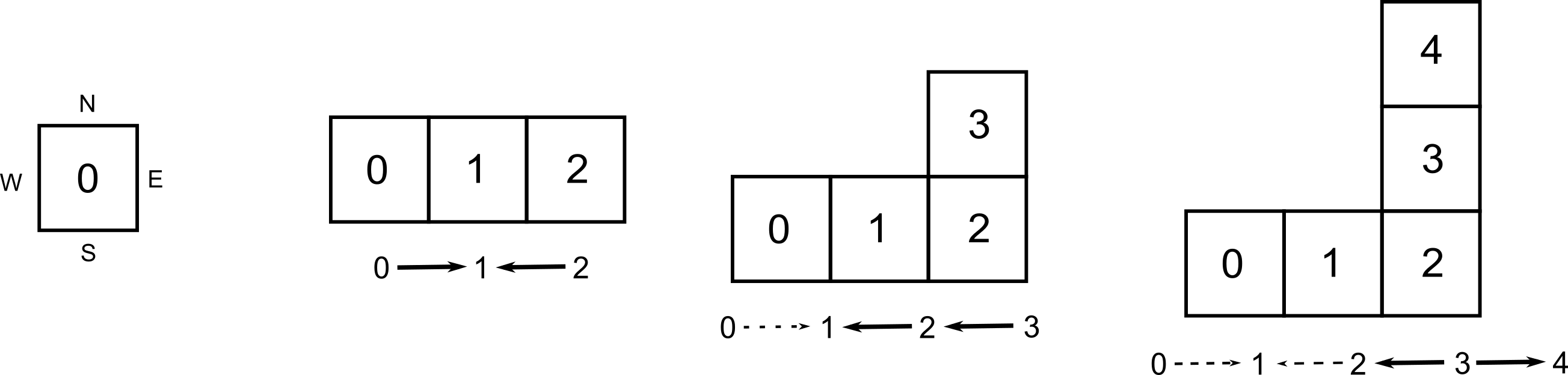}
\end{center}
Note the diagram is split up by the inductive process of stacking tiles on the North or East edge of the previous tile. The subquivers are placed below each snake graph where the arrows are bolded in the quiver that dictate the inductive rule for stacking. 
\end{ex}
\end{subsection}

In the ensuing subsection, we extend this procedure to acyclic quivers of type $D_n$, which leads to a single hexagonal tile in addition to the usual square tiles.

\begin{subsection}{Obtaining a Hexagon-Square Graph from a Quiver} \label{sec:base_graph}
Consider the Dynkin diagram of type $D_n$ where we agree to associate the labels $0,1, \dots, n-1$ to the vertices of the diagram such that the unique degree 3 vertex is labeled $n-3$.

\begin{center}
\begin{tikzpicture}
\node at (-2,0) (0){$0$};
\node at (-1,0) (1){$1$};
\node at (0,0) (2){$2$};
\node at (1.5,0) (dots){$\dots$};
\node at (3,0) (n-4){$n-4$};
\node at (5,0) (n-3){$n-3$};
\node at (7,1) (n-2){$n-2$};
\node at (7,-1) (n-1){$n-1$};
\draw[-, thick] (0) -- (1);
\draw[-, thick] (1) -- (2);
\draw[-, thick] (2) -- (dots);
\draw[-, thick] (dots) -- (n-4);
\draw[-, thick] (n-4) -- (n-3);
\draw[-, thick] (n-3) -- (n-2);
\draw[-, thick] (n-3) -- (n-1);
\end{tikzpicture}
\end{center}

We turn this Dynkin diagram into an acyclic quiver of type $D_n$ by assigning an orientation to each edge of the diagram. We now associate a graph to $Q$ we call the \textbf{base graph}, by associating either a square tile or hexagon tile to each vertex in $Q$. Then, we attach these tiles based on rules dictated by the orientation of the arrow connecting the vertices \cite{reu}. \allowbreak  \vspace{1em}

The unique trivalent vertex labeled $n-3$ corresponds to the unique hexagon tile in our graph. The rest of the vertices in $Q$ are assigned square tiles. For the case when $n=4$, all vertices are incident to the trivalent vertex $n-3$, so we attach the square tiles to the hexagon based on the orientation of the arrow in $Q$. Note that the base graph, by construction, is a planar bipartite graph. Hence, there are two choices for a black and white coloring on the vertices of the base graph. We adopt the convention that if $i \to j$ in the quiver $Q$, we ``see white on the right." For example, the $8$ possible acyclic quivers of type $D_4$ are illustrated superimposed below, with the associated base graphs similarly superimposed. 

\begin{center}
    \includegraphics[scale=.25]{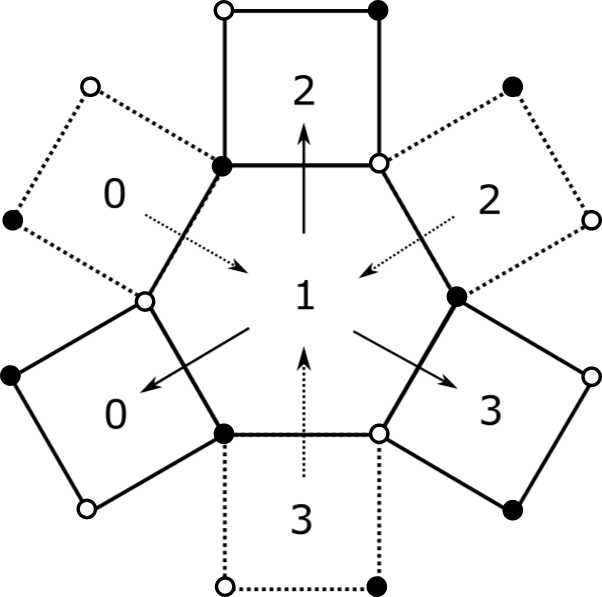}
\end{center}

Using the fact, for $n\geq 4$, that an acyclic $D_n$ quiver contains an acyclic $A_{n-3}$ subquiver induced on vertices $\{0,1,2,\dots, n-4\}$, we build the usual type $A$ snake graph via square tiles, and include this snake graph as an induced subgraph of our base graph.\allowbreak \vspace{1em}

From this point forward, when we refer to the base graph, we implicitly take it with this bipartite coloring. When $n\geq 5$, to attach the remaining tiles, we adopt the type $A$ snake graph convention as described in Section \ref{sec:snake_A}. Below we illustrate some of the possible base graphs for $n=5$ and $6$.  To limit the number of cases, we omit the square tiles associated to $(n-2)$ and $(n-1)$, and assume without loss of generality that the quiver contains the arrow $(n-3) \to (n-4)$ rather than its opposite.
\vspace{1em}

\begin{center}
    \includegraphics[scale=.35]{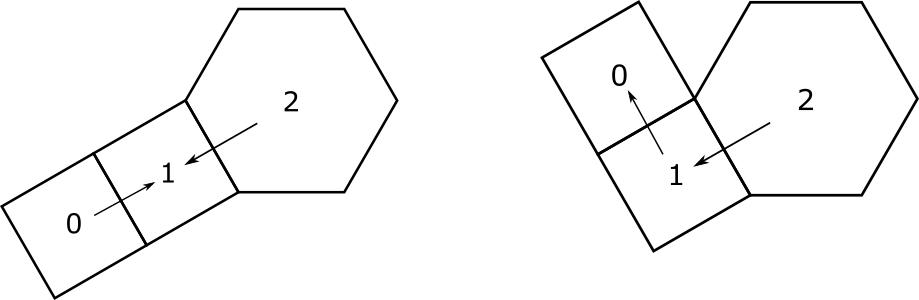}
\end{center}

\begin{center}
    \includegraphics[scale=.35]{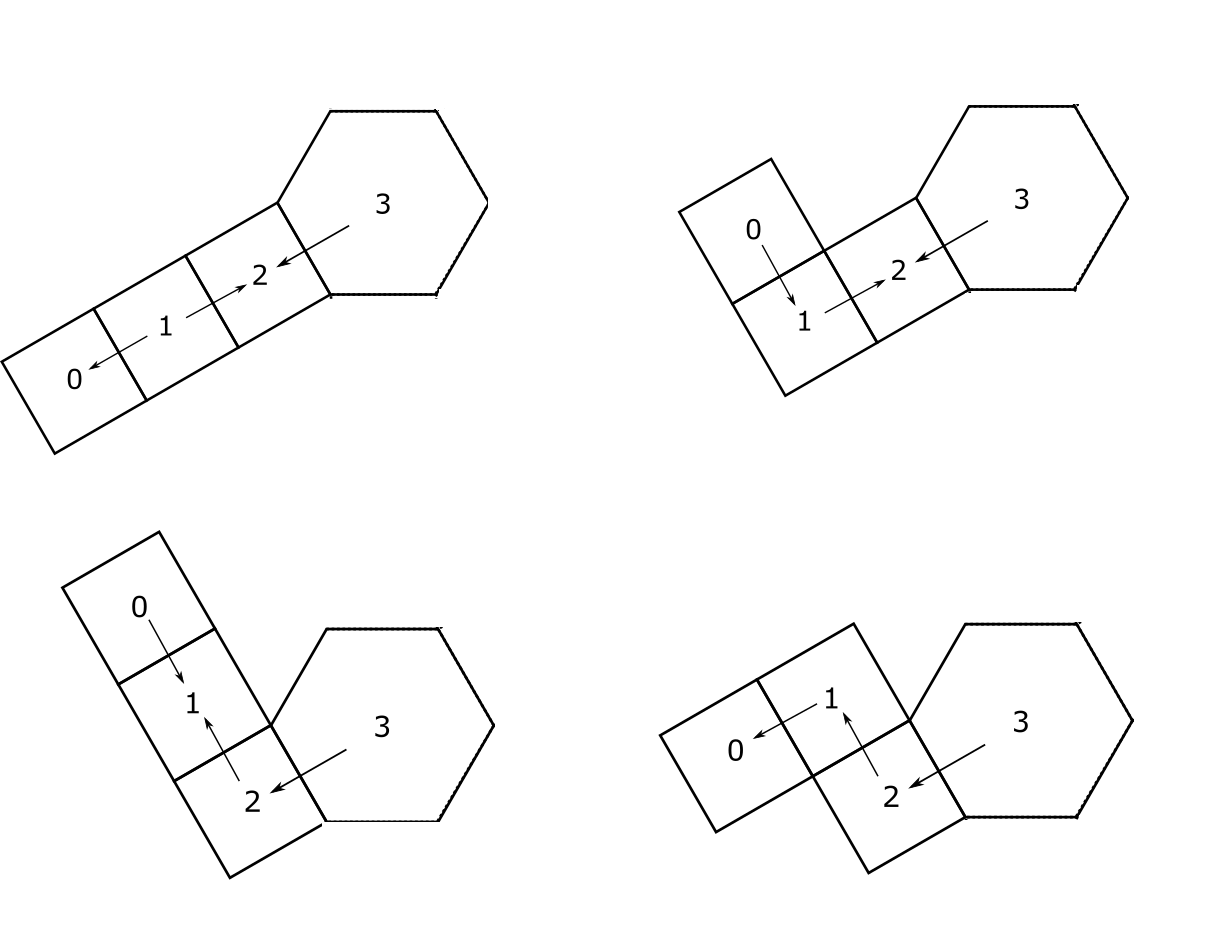}
\end{center}

\begin{defn}
Given a quiver $Q$ of type $D_n$ where the associated quiver is acyclic, the hexagon-square of $Q$ obtained from the above process is called the \textbf{base graph} of $Q$.
\end{defn}

\begin{ex}
Suppose we have the following $D_7$ quiver:

\begin{center}
\begin{tikzpicture}
\node at (-2,0) {$Q=$};
\node at (-1,0) (0){$0$};
\node at (0.5,0) (1){$1$};
\node at (2,0) (2){$2$};
\node at (3.5, 0) (3){$3$};
\node at (5,0) (4){$4$};
\node at (6.5,1) (5){$5$};
\node at (6.5,-1) (6){$6$};
\draw[->, thick] (0) -- (1);
\draw[->, thick] (2) -- (1);
\draw[->, thick] (3) -- (2);
\draw[->, thick] (3) -- (4);
\draw[->, thick] (5) -- (4);
\draw[->, thick] (4) -- (6);
\end{tikzpicture}
\end{center}

By definition of the base graph, mimicking the construction in Example \ref{ex:snake}, the associated hexagon-square graph is:

\begin{center}
    \includegraphics[scale=.3]{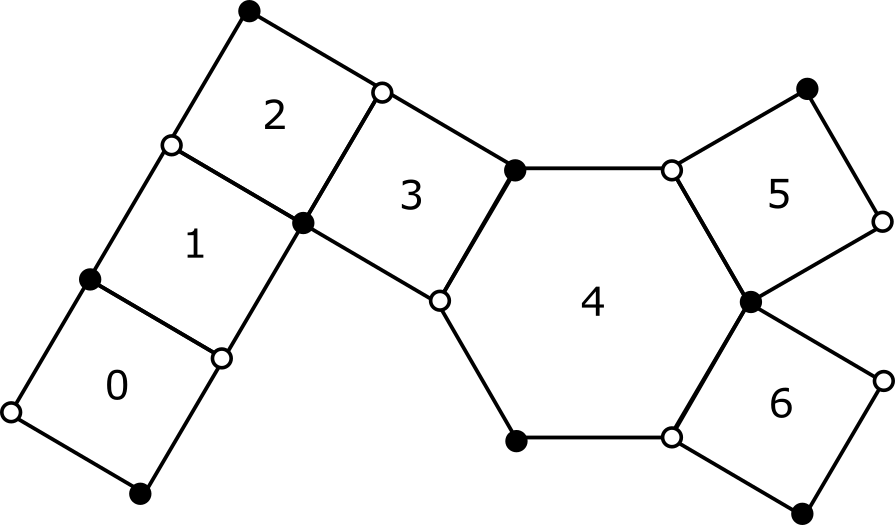}
\end{center}

\end{ex}

\end{subsection}
\begin{subsection}{Associating a Minimal Mixed Dimer Configuration to the Base Graph}
For the following definitions, suppose that $G$ is a planar bipartite graph.

\begin{defn}
 A \textbf{dimer configuration} (also known as a matching) is a subset $D \subset E(G)$ such that every vertex $v \in V(G)$ is contained in exactly one edge $e \in D$.
\end{defn}
\noindent For example, consider the red edges in the following graph:

\begin{center}
    \includegraphics[scale=.25]{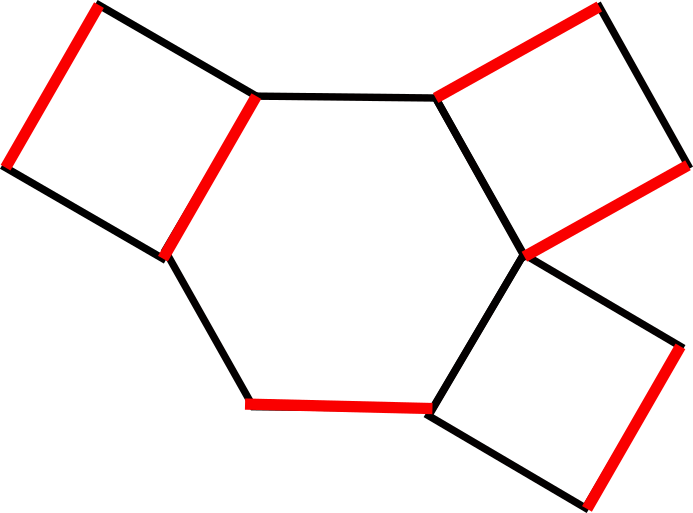}
\end{center}
 
\begin{defn}
A \textbf{double dimer configuration} $D'$ of $G$ is a multiset of the edges of $G$ such that every vertex $v \in V(G)$ is contained in exactly two edges $e,e' \in D'$.
\end{defn}
\noindent For example, consider the red edges in the following graph:

\begin{center}
    \includegraphics[scale=.25]{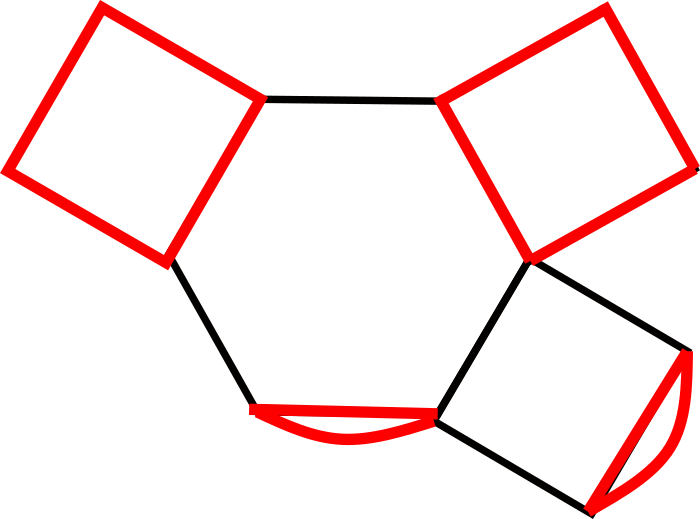}
\end{center}

Now, we aim to associate a ``mixed dimer configuration'' to the base graph obtained in the previous section which, as the name suggests, is a mixture between a dimer configuration and a double dimer configuration.  For example, consider the red edges in the following graph:

\begin{center}
\includegraphics[scale=.2]{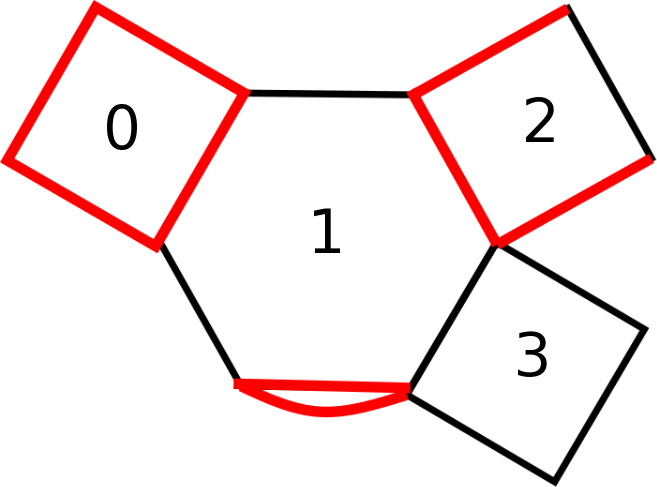}
\end{center}

In order to define mixed dimer configurations that satisfy a certain ``valency condition", we need to introduce another piece of global data other than the quiver $Q$.

\begin{defn} \label{def:mixed_dimer}
Let $\underline{d}$ be an $n$-tuple whose entries are each $0$, $1$, or $2$, i.e. $\underline{d} \in \{0,1,2\}^n$. A \textbf{mixed dimer configuration} $M$ of $G$ is a multiset of the edges of $G$ such that every vertex $v \in V(G)$ is contained in zero, one, or two edges in $M$. Furthermore, we say that $M$ satisfies the {\bf valence condition} with respect to $\underline{d}$ if
\begin{itemize}
    \item Each vertex incident to a tile labeled $i$ with $d_i = 2$ is contained in two edges in $M$.
    \item Each vertex incident to a tile labeled $i$ with $d_i = 1$ is contained in at least one edge in $M$.
\end{itemize}

For example, in the above toy example of a mixed dimer configuration, the associated $4$-tuple would be $\underline{d} = (2,2,1,0)$. Motivated by denominator vectors of cluster variables for cluster algebras of type $D_n$, i.e. of finite type, we focus on the special case where $\underline{d}$ is a positive root of the root system of type $D_n$, which we denote as $\Phi$. 
\end{defn}

\begin{subsubsection}{Defining the minimal mixed dimer configuration $M_-$}
\label{sec:min_match}
Fix some type $D_n$ quiver $Q$ where the associated quiver is acyclic and let $\underline{d} \in \Phi_+$. Let $G$ be the associated base graph with fixed bipartite coloring. For all the entries $d_i \geq 1$, consider the induced subgraph $G_1 \subset G$ using the associated tiles $i$ for $0 \leq i \leq n-1$. Traversing around the boundary of $G_1$ clockwise, distinguish the edges that go from black to white in the bipartite coloring which should notably result in a dimer configuration of $G_1$, call it $M_1$. \allowbreak  \vspace{.5em}

We remark that this is well-defined because the subgraph $G_1 \subset G$ is connected due to the structure of the root $\underline{d} \in \Phi_+$, i.e. all of the positive entries appear such that the tiles in $G$ associated to them are neighboring.\allowbreak  \vspace{1em}

Next, for all the entries $d_j = 2$, consider the induced subgraph $G_2 \subset G_1 \subset G$ using the associated tiles $j$ for $0 \leq j \leq n-1$. Going around the boundary of $G_2$ clockwise, distinguish the edges that go from black to white in the bipartite coloring which should notably result in a dimer configuration of $G_2$, call it $M_2$. \allowbreak  \vspace{1em}

Again, we remark that this is well-defined because the subgraph $G_2 \subset G$ is connected due to the structure of the root $\underline{d} \in \Phi_+$.\allowbreak  \vspace{1em}

Taking the superimposition of $M_1$ and $M_2$, we have some multiset of the edges of our base graph $G$. At this stage, the valence\footnote{
Here, and from now on, when we refer to the ``valence" of a vertex, we refer to the valence with respect to a mixed dimer configuration rather than the valence of the vertex in the base graph.\allowbreak} of each vertex in either 0, 1, or 2 in $M_1 \sqcup M_2$. In particular, for every vertex $v \in G$, we insist the valence of $v$ in $M_1 \sqcup M_2$ is the maximal value of $d_i$, for $i$ in the set of labels for tiles incident to vertex $v$. In particular, the mixed dimer configuration $M_1 \sqcup M_2$ constructed in this way satisfies the valence condition with respect to $\underline{d}$, as defined in Definition \ref{def:mixed_dimer}.

\begin{defn}
The \textbf{minimal mixed dimer configuration} associated to the bipartite planar base graph $G$ and to the type $D_n$ positive root $\underline{d} = (d_0, \dots, d_{n-1}) \in \Phi_+$ is given by 
$$M_-(Q,\underline{d}) := M_1 \sqcup M_2.$$

For brevity, we denote the minimal mixed dimer configuration by $M_- := M_-(Q,\underline{d})$. Sometimes we will instead use the term {\bf minimal matching} for $M_-$. We illustrate this process via an example.
\end{defn}

\begin{ex}
Consider the quiver in Example 2.1 and take the positive root\newline
$\underline{d} = (0,1,1,2,2,1,1)$. Then, the graph $G_1$ will consist of all the tiles except the tile 0, as $d_0 = 0$. When enumerating the boundary edges of $G_1$ oriented black to white clockwise, we obtain $M_1$:

\begin{center}
    \includegraphics[scale=.25]{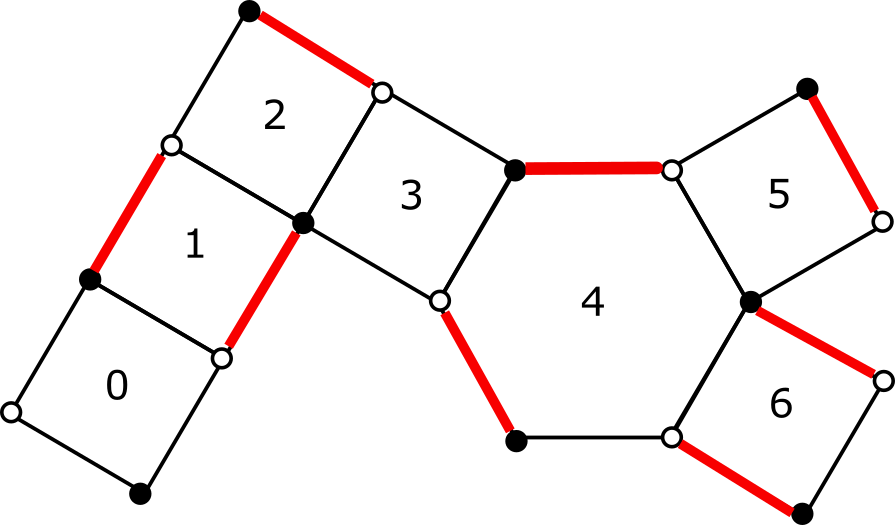}
\end{center}

Since $d_3=d_4=2$, we have that the graph $G_2$ contains tiles 3 and 4. So, we distinguish a double dimer configuration on these tiles by traversing the boundary of tiles 3 and 4 and enumerating the edges that are oriented black to white clockwise. If one of these edges was distinguished in the previous step, distinguish it again with a doubled edge. We obtain $M_-(Q,d)=M_1 \sqcup M_2$:

\begin{center}
    \includegraphics[scale=.25]{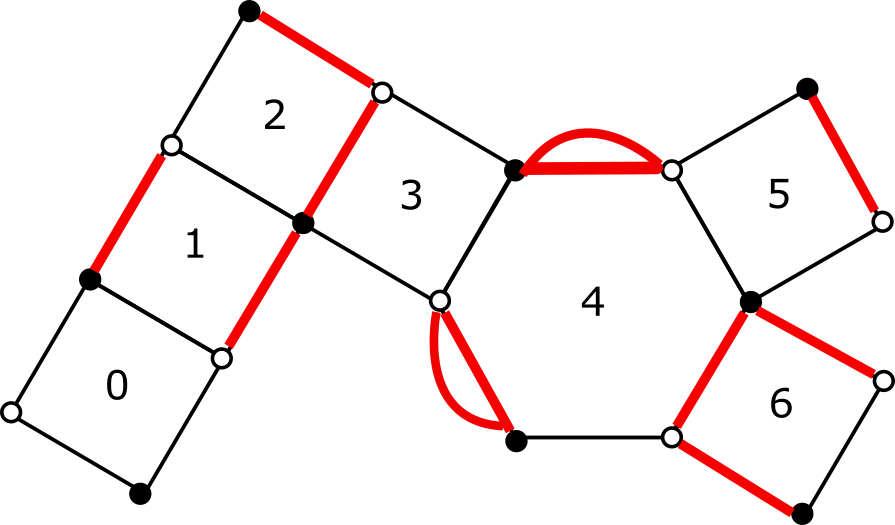}
\end{center}

This is the minimal matching associated to $Q$ and $\underline{d}$.
\end{ex}
\end{subsubsection}
\end{subsection}

\begin{subsection}{Poset of Mixed Dimer Configurations}
\label{sec:poset}
We now create the poset of mixed dimer configurations where each of the mixed dimer configurations appearing in this poset corresponds to a monomial that appears in the $F$-polynomial associated to $\underline{d}$. In order to define this poset, we first define a larger poset, and then refine its set of elements to obtain the poset which we desire.
\allowbreak  \vspace{1em}

The poset relation we will define is inspired by flips of tiles in dimer configurations, in the sense of \cite{Kenyon-Wilson} and \cite{ms}. This local operation is classically given by:

\begin{center}
    \includegraphics[scale=.25]{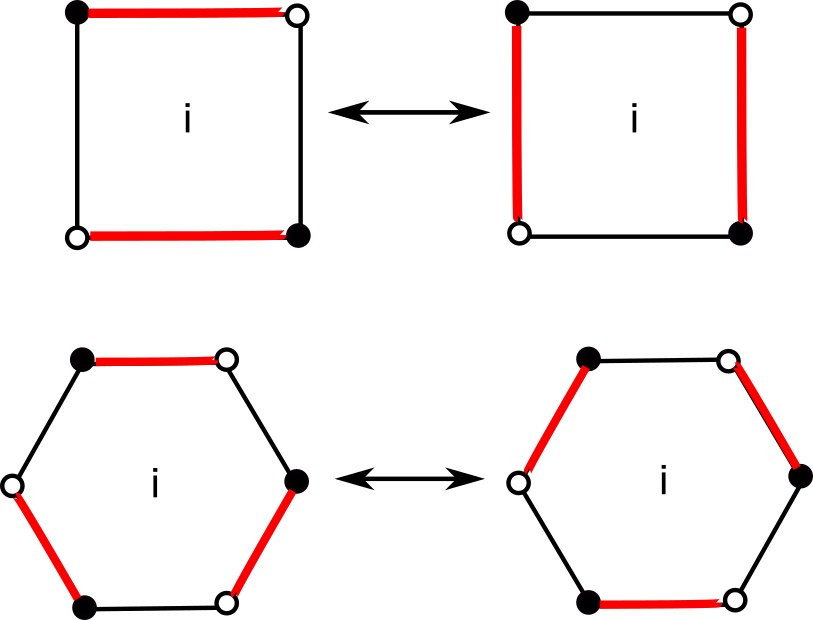}
\end{center}

\begin{rmk}
Not all tiles in a given dimer configuration can be flipped. Consider the following example:

\begin{center}
    \includegraphics[scale=.25]{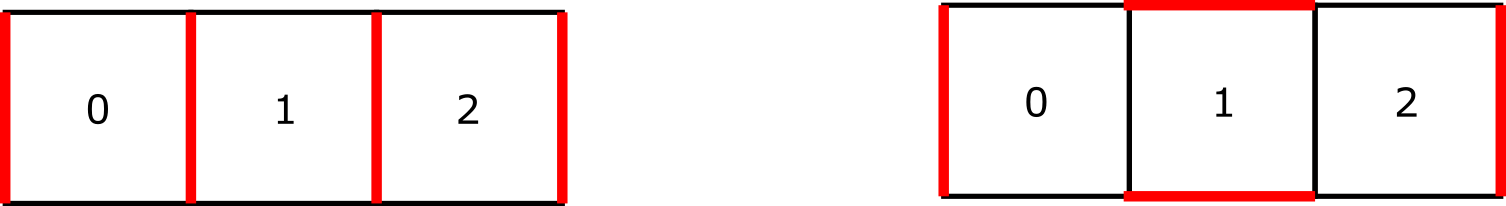}
\end{center}

In the left dimer configuration, we can flip every tile 0, 1, and 2 from the vertical configuration to the horizontal configuration. However, in the right dimer configuration, the only tile we can flip is tile 1, i.e. by replacing the horizontal configuration with the vertical configuration. We cannot flip tile 0 because the edge straddling tile 0 and 1 is not present in this configuration. Similarly, we cannot flip tile 2.
\end{rmk}

To extend this to a mixed dimer configuration, we define flipping as moves that we can perform twice in a row on tiles:

\begin{center}
    \includegraphics[scale=.25]{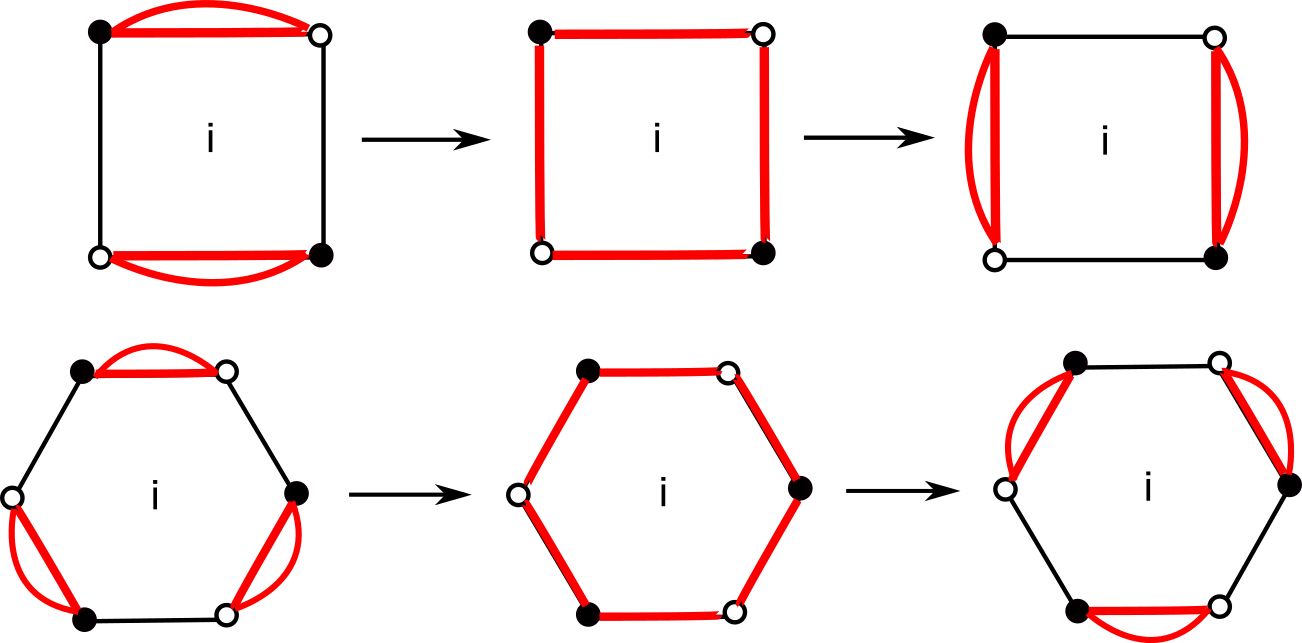}
\end{center}

\begin{rmk}
This modification is necessary to model the $F$-polynomial in type $D$ because there are monomials that appear with terms with an exponent of two i.e. a term like $u_i^2$. In the type $A$ case, no such exponent of two appears which meant that the dimer configurations where flipping is an involution was sufficient.
\end{rmk}

\begin{defn}
We say that the flip of tile $i$ is \textbf{allowable} if we can flip  tile $i$ as in the pictures shown above. More concretely, if all edges going black to white clockwise occur in a mixed dimer configuration on tile $i$, we can exchange these edges for the edges that go white to black clockwise. If for a given edge of the graph, there are multiple copies of that edge in a mixed dimer configuration, then with each flip, we only exchange a single copy of that edge. Implicit in this definition is that the associated $d_i$ is positive.
\end{defn}

\begin{defn}
Let $(\bar{P}, \leq)$ be the {\bf poset of mixed dimer configurations} that satisfy the valence condition and are reachable via a sequence of allowable flips from $M_-(Q, \underline{d})$.  For two such mixed dimer configurations $D$ and $D'$, we say that $D \leq D'$ if there exists a sequence of allowable flips from $D$ to obtain $D'$.
\end{defn}

The poset $\bar{P}$ turns out to have more elements than the number of monomials in the $F$-polynomial associated to $\underline{d}$. Hence, we need to refine the number of elements in our poset by disallowing some sequences of allowable flips. We will let $(P, \leq) \subseteq (\bar{P}, \leq)$ be the subposet of mixed dimer configurations that satisfy the valence condition, are reachable via a sequence of allowable flips from $M_-$, and satisfy the following condition known as being \textbf{node monochromatic}.

\begin{subsubsection}{Node Monochromatic Mixed Dimer Configurations}
\label{sec:nodes}
Fix a type $D_n$ quiver $Q$ and a positive root $\underline{d} = (d_0, \dots, d_{n-1}) \in \Phi_+$. Let $G$ be the associated base graph with fixed bipartite coloring. We now define a condition required of mixed dimer configurations to model the $F$-polynomial. \allowbreak \vspace{1em}

We distinguish three pairs of vertices of $G$ that we call \textbf{nodes} via the following prescription:
\begin{itemize}
    \item On the tile labeled $n-1$, take the two vertices of the square tile that are not incident to the hexagon, call them $u,v \in V(G)$. Color these nodes $u,v$ \textcolor{red}{red}.
    \item On the tile labeled $n-2$, take the two vertices of the square tile that are not incident to the hexagon, call them $w,x \in V(G)$. Color these nodes $w,x$ \textcolor{blue}{blue}.
    \item The last choice of vertices splits into a few cases. Let $d_j$ be a 2 in the positive root $\underline{d}$ of minimal index. \allowbreak  \vspace{1em}
    
    \textbf{Case 1}: If $j-2<0$, then this means that $j=1$. Take the two vertices on tile $0$ that are not incident to tile $1$, call them $y,z \in V(G)$. Color these nodes $y,z$ \textcolor{olive}{green}.\allowbreak  \vspace{1em}
    
    \textbf{Case 2}: Suppose $j \geq 2$. We now have two subcases.
    \begin{itemize}
        \item \textbf{Case 2a.} If the tiles $j-2, j-1, j$ are placed straight, i.e. the orientation in the quiver is given by $j-2 \to j-1 \leftarrow j$ or $j-2 \leftarrow j-1 \to j$, we take the two vertices on tile $j-1$ that are not incident to tile $j$, call them $y,z \in V(G)$. 
        \item \textbf{Case 2b.} If the tiles $j-2, j-1, j$ are placed in a zig-zagging fashion, i.e. the orientation in the quiver is given by $j-2 \leftarrow j-1 \leftarrow j$ or $j-2 \to j-1 \to j$, we choose $y$ on tile $j-1$ as the vertex belonging to the edge that is incident to tile $j-2$, but not tile $j$. We choose $z$ on tile $j-2$ that is ``diagonally across" from $y$. Color these nodes $y,z$ \textcolor{olive}{green}.  In the two figures (above and below), we illustrate our definition of the green nodes.
    \end{itemize}
    
    The placement of the \textcolor{olive}{green} nodes is illustrated in the following figure:\footnote{Note that to make these diagrams as clear as possible, we have fixed an orientation of the arrows between $n-1$, $n-2$ and $n-1$, $n-3$ as we previously defined the red and blue nodes. We have also fixed an orientation between $n-3$ and $n-4$ as it does not affect the definition of the green nodes.} 
    \begin{center}
    \includegraphics[scale=.23]{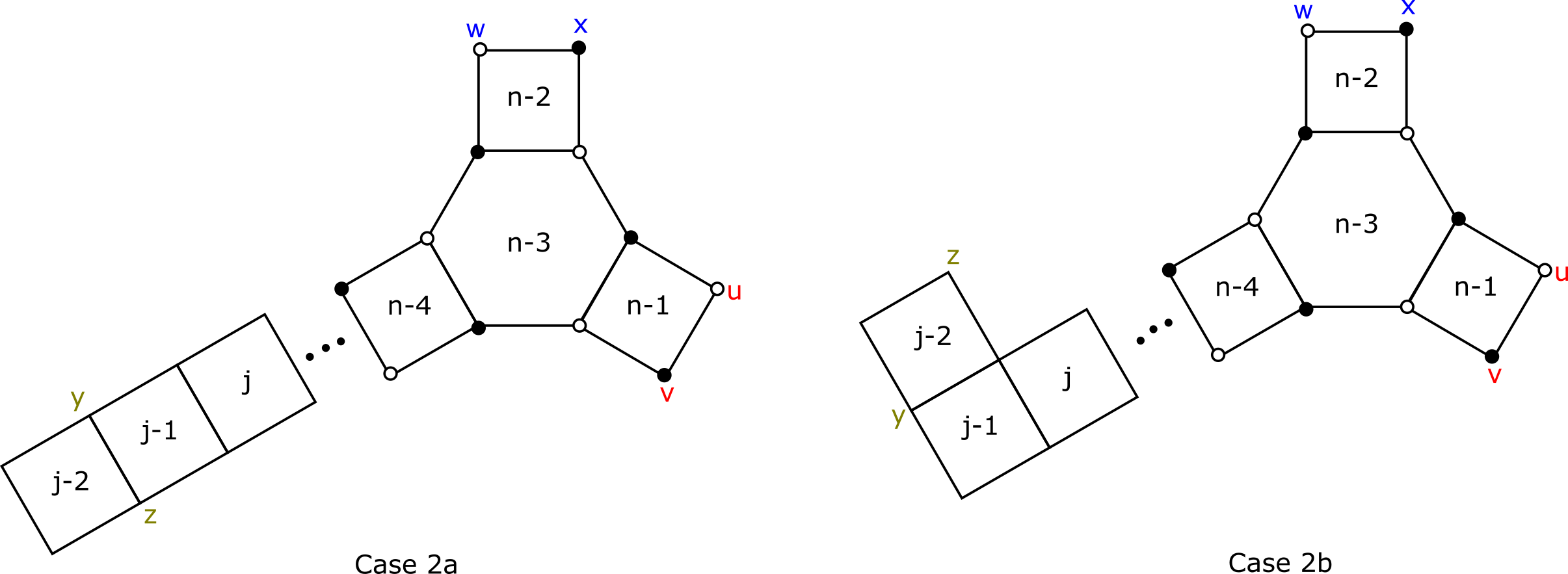}
    \end{center}
\end{itemize}

\begin{defn} \label{defn:nodemonochromatic}
 A mixed dimer configuration $M$ of $G$ is \textbf{node-monochromatic} if any path consisting of edges in $M$ between nodes connects nodes of the same color. If there exists a path consisting of edges in $M$ between nodes of different colors, we say $M$ is \textbf{node-polychromatic}.
\end{defn}

\begin{rmk}
We note that if $\underline{d} \in \{0,1\}^n$, i.e. $\underline{d}$ does not contain a $2$, then any mixed dimer configuration satisfying the valence condition given by $\underline{d}$ is always node-monochromatic.
To see this most easily, we observe that such a mixed dimer configuration is a perfect matching on the subgraph defined by the support of $\underline{d}$, but the distance between any two nodes of different colors is at least two.
\end{rmk}

\begin{rmk}
Our definition of node-monochromatic dimer configurations in Definition \ref{defn:nodemonochromatic} differs from the rules appearing in related research on weighted enumeration of double dimer configurations such as in Jenne's work \cite{helen} or in Kenyon-Wilson's work \cite{Kenyon-Wilson} in the following way.  In \cite{helen} or \cite{Kenyon-Wilson}, the set of nodes is divided into three circularly contiguous sets \textcolor{red}{$R$}, \textcolor{green}{$G$}, and \textcolor{blue}{$B$} in such a way so that no node is paired with a node in the same set. In our model, we also have nodes colored by three circularly continguous sets, but being node-monochromatic requires (rather than forbids) the endpoints of a path are of the same color.
\end{rmk}

\end{subsubsection}
Now that the poset $P$ is well-defined, we aim to show that the minimal matching is in this poset.
\begin{prop}
\label{prop:minimalinP}
$M_-$ is in $P$, i.e. is a mixed dimer configuration on the base graph that satisfies the valence condition and is node-monochromatic.
\end{prop}

\begin{proof}
First note that $M_-$ trivially satisfies the condition that it is reachable by a sequence of allowable flips from $M_-$ by taking the empty sequence. Note that $M_-$ also satisfies the valence condition by construction. So, it suffices to show that $M_-$ is node-monochromatic by showing that there are no paths between nodes of different colors. This is a consequence of the bipartite coloring. By a case-by-case analysis for the acyclic orientations of quivers of type $D_4$, we see that no nodes of different colors can be connected in any orientation of the quiver. The following diagram represents all 8 acyclic orientations of $D_4$ drawn on one graph where the choice of orientation of arrow in the quiver is indicated on top of the graph:

\begin{center}
    \includegraphics[scale=.25]{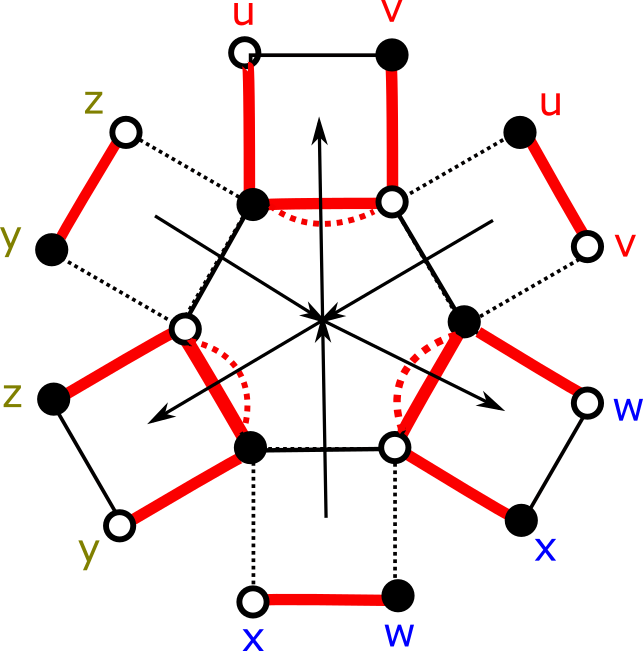}
\end{center}

By the definition of $M_-$, for any $n$, the longest path distinguished in the matching has length 3. This occurs because when enumerating the dimer configuration in $G_1$, we only create paths of length 1. When enumerating edges in $G_2$, we distinguish all clockwise black to white oriented edges on the boundary of $G_2$. If these edges are also boundary edges of $G_1$, then we create a 2-cycle on these edges. If these edges were not boundary edges of $G_1$, they must be interior edges of $G_1$ in which case their inclusion creates a path of length 3 as enumerating this interior edge connects two boundary edges that were previously distinguished.

Suppose that two nodes of different colors can be connected by a path of length 2 in $G$. (Note that no path of length 1 can connect nodes of different colors by definition.) In this case, it suffices to look at $D_4$, as with higher $D_n$, the nodes $y,z$ move farther away from the other nodes and cannot be connected by a path of length 2. Both edges in this path of length two must be boundary edges of $G$, i.e. edges that do not straddle two different tiles in $G$. Moreover, with respect to the full graph $G$, one is colored black to white clockwise and the other must be colored white to black clockwise. By definition of $M_-$, we do not distinguish boundary edges that are oriented white to black clockwise. So, only the edge oriented black to white clockwise will appear in $M_-$ which means that this path cannot be contained in $M_-$. \allowbreak  \vspace{1em}

Now suppose that there is a path of length 3 in $G$ connecting nodes of different colors.  In this case, it suffices to look at $D_5$, as with higher $D_n$, the nodes $y,z$ move farther away from the other nodes and cannot be connected by a path of length 3. In order to connect nodes of different colors, this path would have to traverse over a white to black clockwise edge on the boundary of $G_2$. This is because in any orientation of $D_5$, paths cannot be formed using the white to black clockwise oriented edges with respect to the hexagon in $M_-$. \allowbreak \vspace{1em}

For $D_5$, there are two positive roots that make it potentially possible for nodes of different color to be connected; either $\underline{d}=(1,2,2,1,1)$ or $\underline{d}=(1,1,2,1,1)$. The following diagram represents the minimal matching when $\underline{d}=(1,2,2,1,1)$, i.e. Case 1 in Section \ref{sec:nodes}. In all possible orientations of $D_5$, paths in $M_-$ do not connect nodes of different colors\footnote{Just like in the $D_4$ case above, we superimpose together the various possible base graphs for $D_5$, but use dotted edges to illustrate the contrasts.}.  

\begin{center}
    \includegraphics[scale=.2]{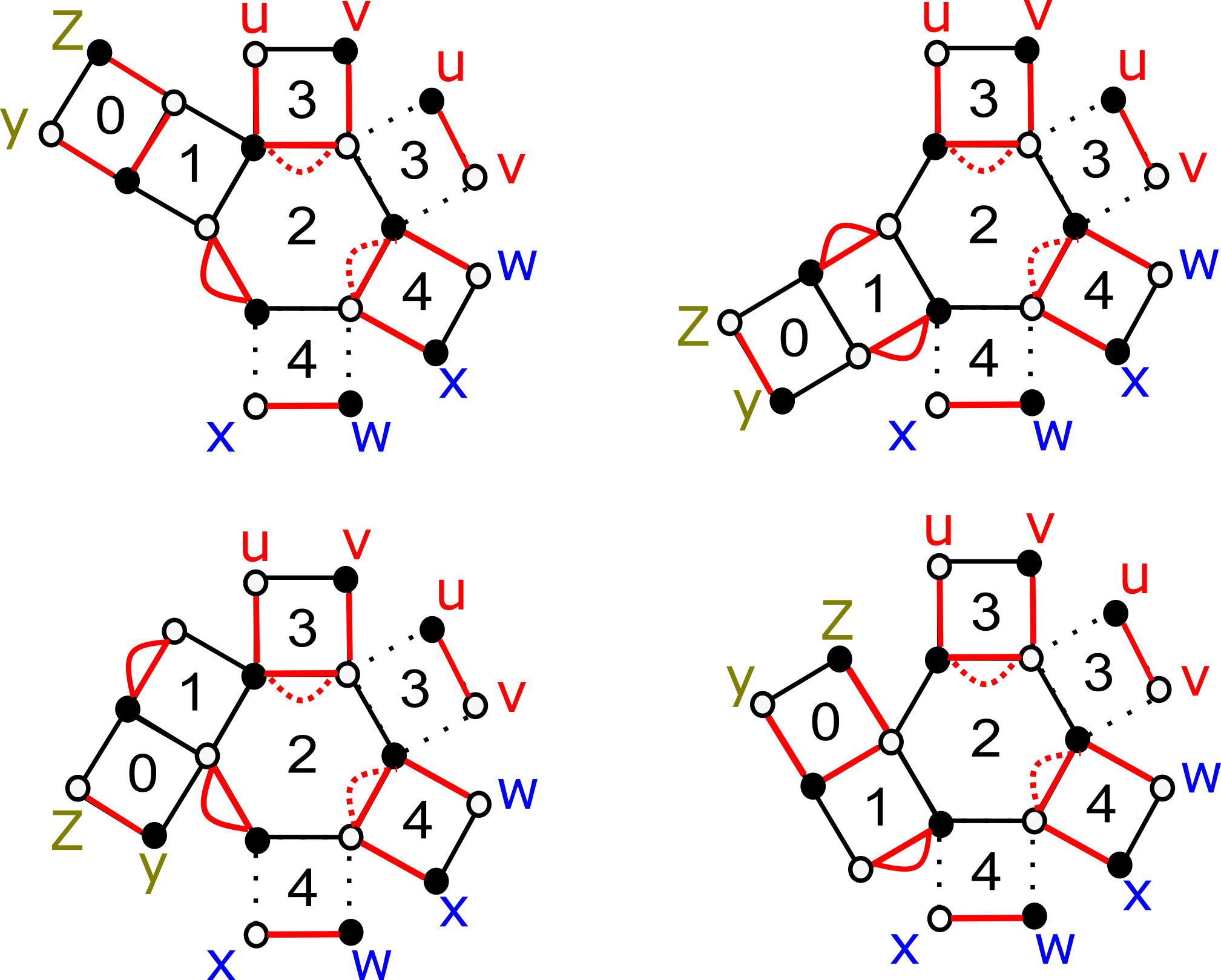}
\end{center}

The next diagram represents the minimal matching when $\underline{d}=(1,1,2,1,1)$ i.e. Cases 2a and 2b in Section \ref{sec:nodes}. In all possible orientations of $D_5$, paths in $M_-$ do not connect nodes of different colors. 

\begin{center}
    \includegraphics[scale=.2]{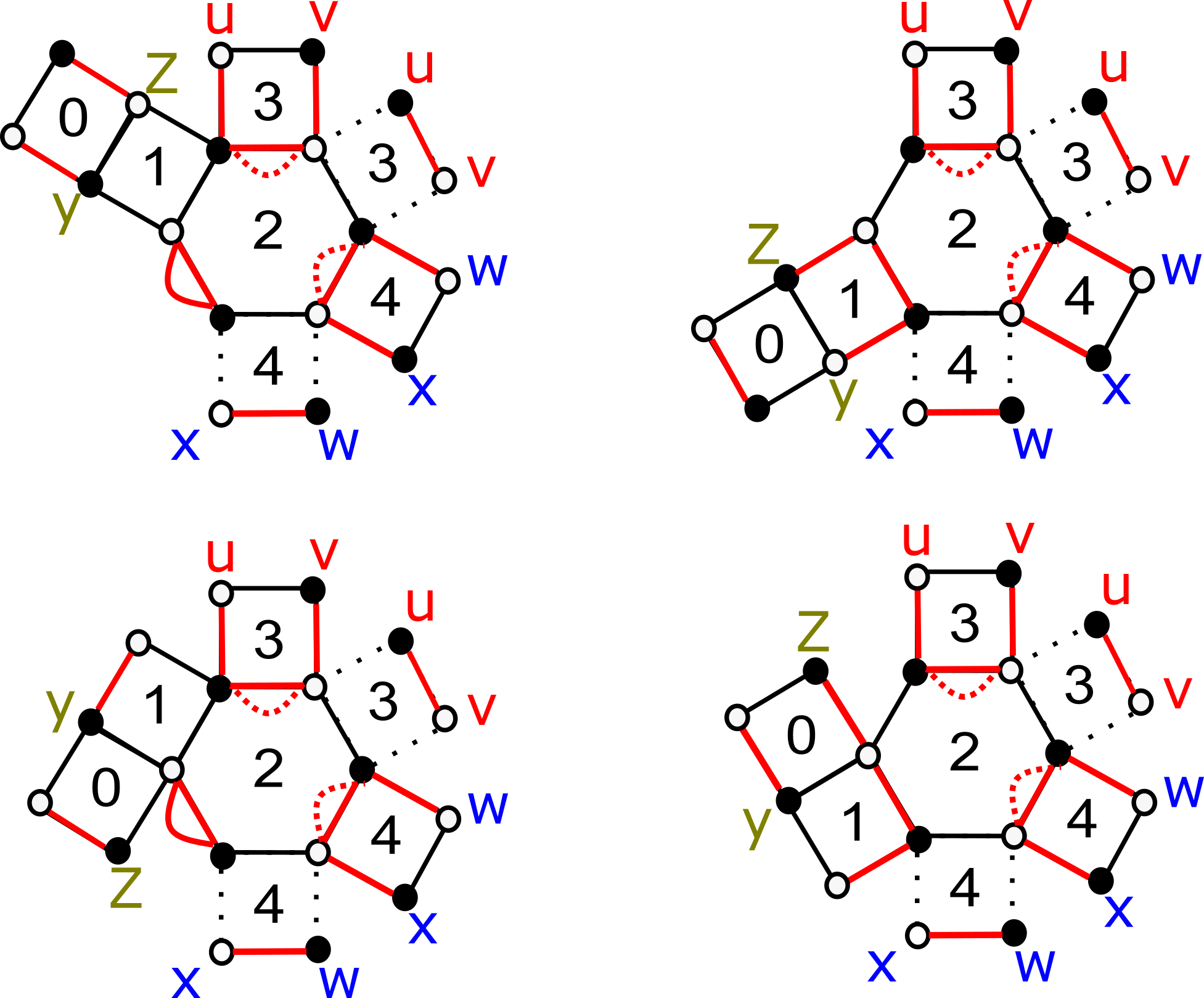}
\end{center}

So, by definition of $M_-$, no boundary white to black clockwise edge is distinguished in $G_2$ which means this path cannot be contained in $M_-$. Thus, $M_-$ is node-monochromatic.

\end{proof}

\end{subsection}

We show that the elements in $(P,\leq)$ give the nonzero monomials in the $F$-polynomial associated to $\underline{d}$.

\begin{subsection}{Main Theorem}
\begin{thm}
\label{thm:main}
Given an acylic quiver $Q$ of type $D_n$ and a choice of positive root $\underline{d} \in \Phi_+$ for the corresponding root system, we let $F_{\underline{d}}$ denote the $F$-polynomial corresponding to the cluster variable with denominator vector $\underline{d}$.  This expression is based on the appropriate cluster algebra of type $D_n$ and assuming an initial seed defined by the choice of quiver $Q$ and the standard initial cluster of $\{x_1,x_2,\dots, x_n\}$. \allowbreak \vspace{1em}

Furthermore, let $M_- = M_-(Q,\underline{d})$ be the minimal matching, as defined in Section \ref{sec:min_match} and let $P$ be the poset of mixed dimer configurations that satisfy the valence condition, are reachable via a sequence of allowable flips from $M_-$, and satisfy the node monochromatic condition. Then the expansion of $F$-polynomial $F_{\underline{d}}$ can be expressed as a weighted multi-variate rank-generating function on the poset $(P, \leq)$ determined by $M_-$:

$$F_{\underline{d}} = \sum_{D \in P} 2^c u_{0}^{t_0} u_1^{t_1} \cdots u_{n-1}^{t_{n-1}},$$
where the sum is taken over mixed dimer configurations $D$ obtained by flipping tile $i$ $t_i$-times (keeping track of multiplicities) and $c$ is the number of cycles enclosing a face in $D$.
\end{thm}
We provide an example before proceeding to the proof of Theorem {\ref{thm:main}}.

\begin{ex}
\label{d6example}
Suppose that we have the following $D_6$ quiver:

\begin{center}
\begin{tikzpicture}
\node at (-1,0) {$Q=$};
\node at (0,0) (0){$0$};
\node at (1.5,0) (1){$1$};
\node at (3,0) (2){$2$};
\node at (4.5,0) (3){$3$};
\node at (6,1) (4){$4$};
\node at (6,-1) (5){$5$};
\draw[->, thick] (1) -- (0);
\draw[->, thick] (2) -- (1);
\draw[->, thick] (3) -- (2);
\draw[->, thick] (4) -- (3);
\draw[->, thick] (3) -- (5);
\end{tikzpicture}
\end{center}

Suppose that we take $\underline{d} = (1,1,2,2,1,1) \in \Phi_+$. Then, we have that the $F$-polynomial associated to $\underline{d}$ is given by:

\begin{align*}
    F_{\underline{d}} &= 1+u_0+u_2+u_5+u_0u_1+u_0u_2+u_0u_5+u_2u_5\\
    &+2u_0u_1u_2+u_0u_1u_5+u_0u_2u_5+u_2u_3u_5+2u_0u_1u_2u_5+u_0u_1u_2^2\\
    &+u_0u_1u_2u_3+u_0u_2u_3u_5+2u_0u_1u_2u_3u_5+u_0u_1u_2^2u_5+u_0u_1u_2^2u_3\\
    &+u_0u_1u_2u_3u_4u_5+2u_0u_1u_2^2u_3^2u_5+u_0u_1u_2^2u_3u_4u_5+u_0u_1u_2^2u_3^2u_5+u_0u_1u_2^2u_3^2u_4u_5
\end{align*}

The poset $P$ is Figure 1 on the following page where tiles are shaded grey to emphasize when they have been enclosed by a cycle.

\newpage

\begin{figure}[H]
    \centering
    \includegraphics[scale=.225]{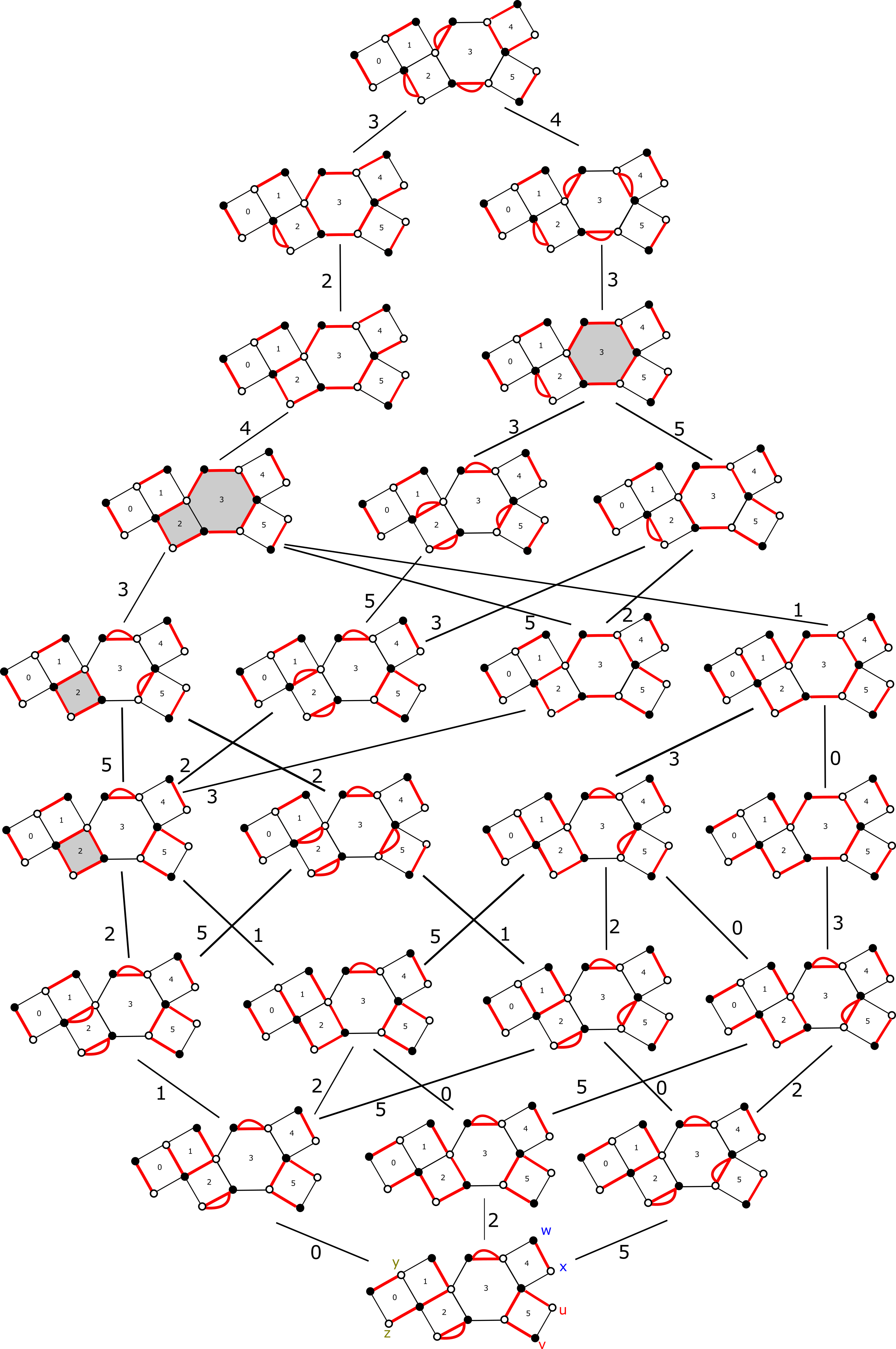}
    \caption{The poset associated to Example \ref{d6example}.}
\end{figure}

\begin{rmk}
\label{rem:Exa}
We see that we do not flip at all tiles at which we can perform allowable flips. It is crucial that these mixed dimer configurations stay node-monochromatic. For example, if we analyze the set of allowable flips from the minimal matching $M_-$, after flipping tile 2, an allowable flip can occur at the hexagon, tile 3. However, the monomial $u_2u_3$ does not appear in $F_{\underline{d}}$. With our choice of nodes, we see that by flipping tile 3, we connect nodes of different colors i.e. there is a path between the green node $z$ to the red node $v$. This is why the mixed dimer configurations in our poset must be node-monochromatic.
In Section \ref{sec:dist_latt}, we show a poset-theoretic consequence of this restriction.
\end{rmk}

\end{ex}
\end{subsection}

\begin{subsection}{The poset of node-monochromatic mixed dimer configurations is not distributive}
\label{sec:dist_latt}
After defining our poset of (node-monochromatic) mixed dimer configurations, which we denoted as $(P, \leq)$, it is of independent combinatorial interest to investigate its properties. Our poset of mixed dimer configurations forms a lattice since the unique minimal element corresponds to the constant term 1 in the $F$-polynomial and the unique maximal element corresponds to $\underline{u}^{\underline{d}}$ in the $F$-polynomial. \allowbreak \vspace{1em}

Recall a lattice is distributive if the operations of meet and join distribute over one another i.e. for any $A,B,C \in P$, we have 

$$A \vee (B \wedge C) = (A \vee B) \wedge (B \vee C).$$

In the Type A case, Theorem 5.2 in \cite{mswbases}, based on \cite[Section 3]{propplattice}, states that the snake graph poset of perfect matchings form a distributive lattice. It turns out that this is not the case for our poset of mixed dimer configurations as it contains a forbidden sublattice\footnote{On this sublattice, we have the property that the meets and joins of pairs of elements agree with their images as they would be in the original lattice.}  on 5 vertices that form a pentagon shape. Consider the following counterexample deduced from Figure \ref{fig:laurent}: 

\begin{center}
    \includegraphics[scale=.125]{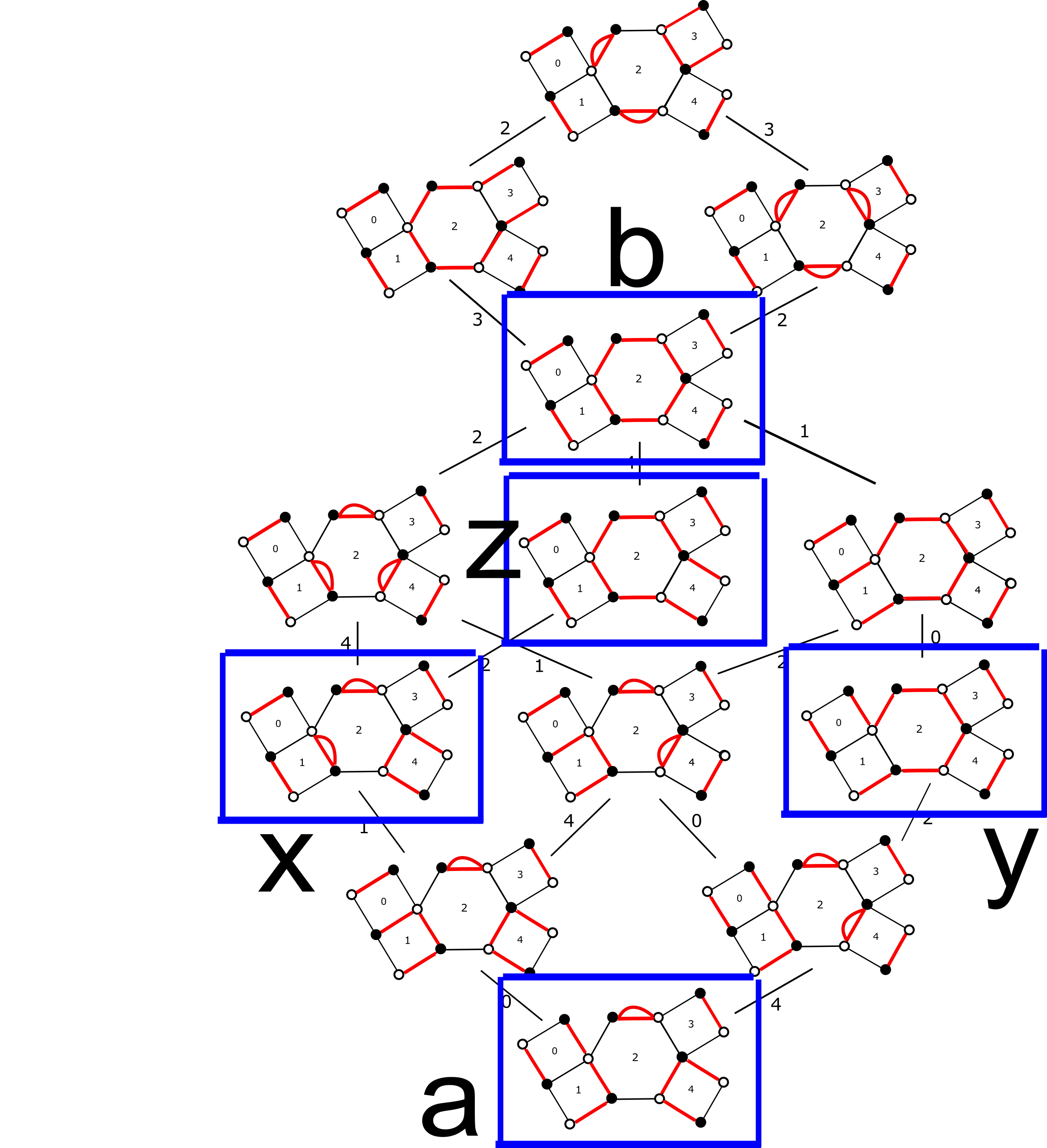}
\end{center}

Note that $x \vee (y \wedge z) = x$ whereas $(x \vee y) \wedge (x \vee z) = z$, Hence, our poset of mixed dimer configurations fails to be distributive.
\end{subsection}

\end{section}
\begin{section}{Proof of Theorem \ref{thm:main}} \label{sec:proof}
In order to prove that this mixed dimer configuration model indeed gives the $F$-polynomial, we construct bijections between mixed dimer configurations and tuples of integers which come from a model of Thao Tran \cite{tran}.
\begin{subsection}{Thao Tran's Model}
Define a partial order $\geq$ on $\zz^n$ via 
$$\underline{a} \geq \underline{a'} \text{ if } \underline{a} - \underline{a'} \in \zz_{\geq 0}^n.$$

\begin{defn} {\cite[Definition 4.1]{tran}}
Fix $n\geq 4$, and let $Q$ denote an acyclic quiver of type $D_n$.  Further, fix
$\underline{d} \in \Phi_+$ , a positive root of type $D_n$, and $\underline{e} = (e_0, \dots, e_{n-1}) \in \zz^n$ with $0 \leq \underline{e} \leq \underline{d}$. An arrow $i \to j$ in $Q$ is \textbf{acceptable} with respect to the pair $(\underline{e}, \underline{d})$ if 

$$e_i-e_j \leq \max(d_i-d_j,0).$$ 

An arrow $i \to j$ is \textbf{critical} with respect to the pair $(\underline{e}, \underline{d})$ if either 

$$(d_i,e_i) = (2,1) \text{ and } (d_j,e_j) = (1,0)$$
or 
$$(d_i,e_i) = (1,1) \text{ and } (d_j,e_j) = (2,1).$$

Let $S$ be the induced subgraph of $Q$ on the set of vertices $\{i~:~(d_i,e_i) = (2,1)\}$. For a connected component $C$ of $S$, define 
$$\nu(C) = \#\{\text{critical arrows having a vertex in }C\}.$$
\end{defn}

\begin{thm}
{\cite[Theorem 4.2]{tran}}
\label{thm:Tran}
Fix $\underline{d} \in \Phi_+$, $\underline{e} \in \zz^n$. For the cluster algebra defined by the initial quiver $Q$ (acyclic of type $D_n$), the coefficient of the monomial $u_0^{e_0} \cdots u_{n-1}^{e_{n-1}}$ in the $F$-polynomial associated to $\underline{d}$ is nonzero if and only if all of the following conditions, which we refer to as {\bf Tran's conditions} in the sequel, are satisfied:

\begin{enumerate}
    \item $0 \leq \underline{e} \leq \underline{d}$
    \item all arrows in $Q$ are acceptable
    \item $\nu(C) \leq 1$ for all connected components $C$ of $S$, where $S$ is the induced subgraph as defined above.
\end{enumerate}
The coefficient of $u_{0}^{e_0} \cdots u_{n-1}^{e_{n-1}}$ is $2^c$, where $c$ is the number of connected components $C$ such that $\nu(C) = 0$.
\end{thm}
 \begin{rmk}
Although \cite{tran} proves Theorem \ref{thm:Tran} in more generality, i.e. for other Dynkin types, we focus on her proof in case of Type $D_n$ quivers. For this, she utilizes \cite[Theorem 4.7]{tran}, whose proof involves the Quiver Grassmanian, denoted $Gr_{\underline{e}}(M)$, which is the variety of all subrepresentations of a quiver representation $M$ of dimension $\underline{e}$. It is a closed subvariety of a product of Grassmannians. Her proof involves computing the $F$-polynomial by calculating Euler-Poincar\'e characteristics for multiple quiver Grassmannians ranging over various vectors $\underline{e}$, relying on \cite{DWZ2} for the correspondence between these two objects. To compute the Euler-Poincar\'e characteristic of $Gr_{\underline{e}}(M)$, Tran shows that conditions (1), (2), and (3) exactly give subrepresentations of dimension $\underline{e}$ for which the corresponding Euler-Poincare characteristic is nonzero. This provides the computation of the support of the $F$-polynomial.
 \end{rmk}

\end{subsection}
\begin{subsection}{Bijection of Models}

\begin{subsubsection}{Monomials} \label{sec: monomials}

\begin{thm}
\label{thm:biject}
Let $Q$ be an acyclic quiver of type $D_n$ and let $\underline{d} \in \Phi_+$. There is a bijection between
$$\{\text{mixed dimer configurations }D \text { in } P\} \stacklongleftrightarrow{\sim} 
\{\text{vectors } \underline{e} \text{ satisfying Tran's conditions}\}.$$
\end{thm}

This bijection is described as a map in both directions. The map sending $\underline{e} \longrightarrow D$ is given by performing a particular set of flips from $M_-$. In order to describe this direction of the bijection, we weight the edges of the base graph $G$ (which is associated to some quiver $Q$ and positive root $\underline{d}$) via
$$w(e) := \#\{ \text{ edges on }e \text{ in } M_-\}$$
for all $e \in E(G)$.
     
\begin{defn}
\label{defn:weightedflip}
The weights of the edges are transformed via the following prescription after flipping a tile: 
\begin{center}
    \includegraphics[scale=.3]{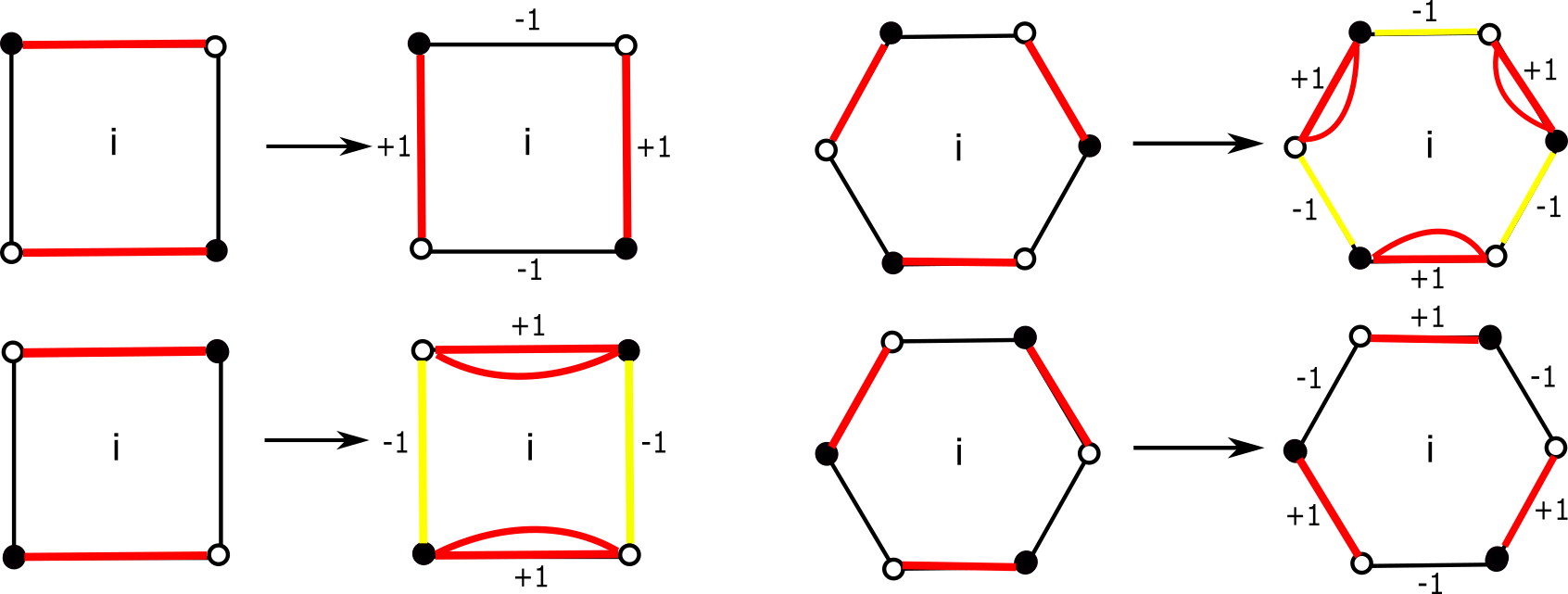}
\end{center}
This is known as a \textbf{weighted flip}.
\end{defn}
    
\begin{rmk}
Implicit in this definition is the fact that our weights will stay non-negative when we perform what we have previously called ``allowable flips." However, in the context of this direction of the bijection, we will allow for weighted flips at any tile; meaning that we now allow the weight of an edge to be negative. To emphasize the deficits which arise when we flip edges that we previously could not in the ``allowable'' sense, we distinguish edges of negative weight in yellow and call them ``antiedges." By introducing these weights and allowing them to be negative, we simplify the noncommutativity of flipping sequences of tiles.
\end{rmk}
     
\begin{thm}
\label{thm:EtoD}
Given $(Q,\underline{d}, \underline{e})$ where $\underline{e}$ satisfies  Tran's conditions, i.e. (1), (2), and (3) of Theorem \ref{thm:Tran},
there exists a unique mixed dimer configuration $D$ in poset $P$  (see Section \ref{sec:poset})
such that $D$ is associated to this $\underline{e}$. Further, the mixed dimer configuration $D$ is obtained by the following procedure: 
         \begin{enumerate}
         \item Given $(Q,\underline{d})$, construct the corresponding base graph $G$ and the minimal matching $M_-$ associated to $\underline{d}$ via the procedure described in Section \ref{sec:base_graph}.
         Weight the edges of $G$ by $w(e)$ for $e \in E(G)$.
         \item From this weighted minimal matching, take the positive entry $e_{i_1} > 0$ in $\underline{e}$ of minimal index. Flip tile $i_1$ $e_{i_1}$ number of times and transform its edge weights as prescribed in Definition \ref{defn:weightedflip} with respect to the black and white coloring of the vertices to arrive at the mixed dimer configuration $D_1$. (Note that if this flip (or flips) were not ``allowable," then we will obtain ``antiedges", i.e. edges with negative weight we highlight in yellow.)
         \item From $D_1$, take the next positive entry $e_{i_2} > 0$ in $\underline{e}$ of minimal index. Flip tile $i_2$ $e_{i_2}$ number of times and transform its edge weights as prescribed in Definition \ref{defn:weightedflip} to arrive at  the mixed dimer configuration $D_2$.
         \item Iterate this process until we have exhausted all positive entries in $\underline{e}$. The resulting mixed dimer configuration $D$ will only have non-negative weights, i.e. no yellow edges will remain. Moreover, it will be an element of $P$. The resulting configuration $D$ is the mixed dimer configuration associated to $\underline{e}$.
     \end{enumerate}
    \end{thm}
    
Before proving this theorem, we compute an example.

\begin{ex}
\label{ex:etoD}
Suppose $\underline{d} = (1,1,2,2,1,1)$ and $Q$ is superimposed on the first mixed dimer configuration. Let $\underline{e} = (0,0,1,2,1,1)$ which we see satisfies conditions (1)-(3) in \cite{tran}. Performing the flips in left to right order as prescribed in Theorem \ref{thm:EtoD}, we obtain the following sequence of weighted flips

\begin{center}
    \includegraphics[scale=.25]{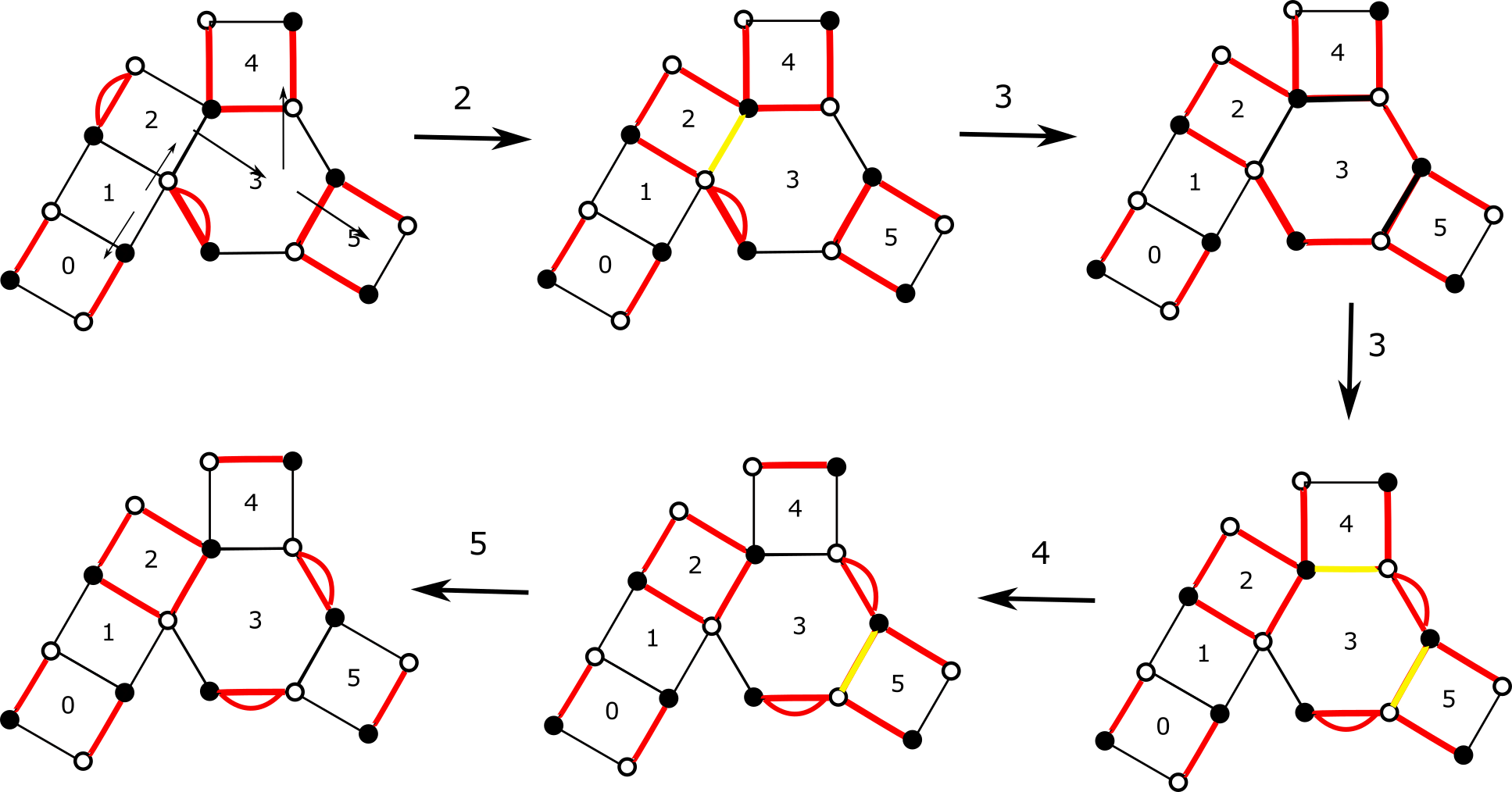}
\end{center}
where yellow edges have weight $-1$ that we refer to as ``antiedges" and the last dimer configuration $D$ is the resulting mixed dimer configuration. Note that, in this example, we could flip tiles in the order 3,2,4,5,3 to obtain a sequence of allowable flips yielding $D$. 
\end{ex}

In order to prove the above statement, we will utilize the following lemmas:

\begin{lemma}
\label{lemma:MijNij}

Let $\underline{d} \in \Phi_+$ and let $0 \leq \underline{e} \leq \underline{d}$. Suppose that $i \to j$ in $Q$ and for any $0 \leq i \leq n-1$, let $D$ be the mixed dimer configuration obtained by flipping tile $i$ $e_i$ number of times from $M_-$. Let $m_{i,j}$ be the number of edges distinguished in $D$ on the edge between tiles $i$ and $j$. Then
     
     $$m_{i,j} = \max(d_i-d_j,0) + (e_j-e_i) =: n_{i,j}.$$
\end{lemma}
     
\begin{proof}
In order to prove this formula holds, we proceed by induction on $|\underline{e}| := \sum_{k=0}^{n-1} e_k$. When $|\underline{e}| = 0$, i.e. $\underline{e} = (0,0, \dots, 0)$, the associated mixed dimer configuration is $M_-$. By definition of $M_-$, we distinguish boundary edges in $G_1$ and $G_2$ that are oriented black to white clockwise with respect to the bipartite coloring. This means that we only distinguish internal edges when $G_1, G_2$ are strict subsets of $G$. Suppose that $d_i > 0$, but $d_j =0$, i.e. $i \in G_1$ and possibly $G_2$ whereas $j \notin G_1$ or $G_2$. If $i \to j$ in $Q$, then locally we have that the edge straddling $i$ and $j$ is oriented black to white clockwise with respect to the tile $i$. So, this gives that $m_{i,j} = d_i$ - as we will distinguish this edge in $G_1$. If $d_i = 2$, we will distinguish it once again. Therefore, in the case that $i \to j$, we've shown
     
     $$n_{i,j} = \max(d_i-d_j,0) = d_i = m_{i,j}.$$
     
If $j \to i$, then we locally have that the edge straddling $i$ and $j$ is oriented white to black clockwise with respect to the tile $i$. So, this gives that $m_{j,i} = 0$ as we will never distinguish  this edge in $G_1$ or $G_2$. Therefore, in the case that $j \to i$, we've shown
     
     $$n_{j,i} = \max(d_j-d_i,0) = 0 = m_{j,i}$$
     
Thus, when $|\underline{e}|=0$, our formula holds. Suppose up to $k \in \nn$, we have that when $|\underline{e}| = k$, our formula holds. Now, suppose $|\underline{e}| = k+1$, and choose $\tilde{\underline{e}}$ with $|\tilde{\underline{e}}|=k$, such that $\underline{e}$ can be derived from $\tilde{\underline{e}}$ by adding $1$ to the $i^{\text{th}}$ entry of $\tilde{\underline{e}}$. On the level of the mixed dimer configuration, this means we flipped tile $i$ from the mixed dimer configuration associated to $\tilde{\underline{e}}$ to obtain the mixed dimer configuration associated to $\underline{e}$.\allowbreak \vspace{1em}

Suppose $i \to j$ in $Q$. Note that by definition of the weighted flip, we have that the weight of the edge straddling tiles $i$ and $j$ is oriented black to white clockwise on tile $i$ which gives us that the weight decreases by 1. So, we have that $m_{i,j}$ has decreased by 1 and since $e_i$ has increased by 1, we have 
     
     $$n_{i,j} = \max(d_i-d_j,0) + (e_j-(e_i+1)) = \max(d_i-d_j,0) + (e_j-e_i) - 1.$$
     
Therefore, performing a flip at tile $i$ decreased both $n_{i,j}$ and $m_{i,j}$ by 1. If $j \to i$ in $Q$, the argument is similar.
\end{proof}
     
\begin{cor}
\label{cor:accnonneg}
The values $n_{i,j} := \max(d_i-d_j,0) + (e_j-e_i)$ all satisfy $n_{i,j} \geq 0$ if and only if all weights on interior edges on the mixed dimer configuration $D$ associated to $e$ are nonnegative.  In particular, all arrows $i \to j$ are acceptable with respect to $(\underline{d}, \underline{e})$ if and only if all weights on interior edges are nonnegative on the mixed dimer configuration $D$ associated to $e$.
\end{cor}
     
\begin{lemma}
\label{lemma:MinfinityN}
(Analogous formula to Lemma \ref{lemma:MijNij} with boundary edges). Let $\underline{d} \in \Phi_+$ and let $0 \leq \underline{e} \leq \underline{d}$. For any $0 \leq i \leq n-1$, let $D$ be the mixed dimer configuration obtained by flipping tile $i$ $e_i$ number of times from $M_-$. Let the outer face of our graph $G$ be indexed by $\infty$. We assign an arrow to each of the boundary edges of $G$ where the orientation of this arrow depends on the bipartite coloring of $G$ following the convention that we ``see white on the right." For example, 
     
\begin{center}
         \includegraphics[scale=.35]{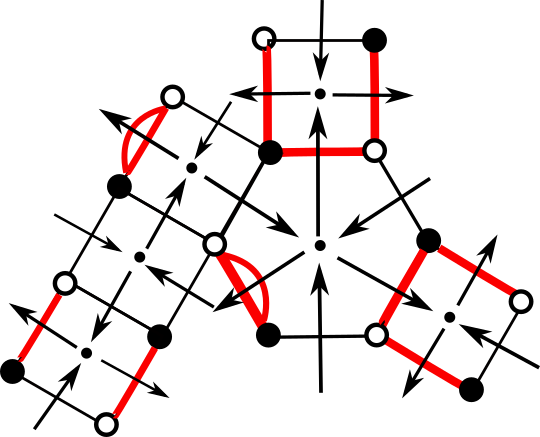}
\end{center}
We assign a weight to the edge $\alpha$ on tile $i$ as follows:
     
     $$n_{i,\infty}(\alpha) = \max(d_i,0) - e_i ~~~\text{if }i \to \infty \text{ about } \alpha$$
     $$n_{\infty,i}(\alpha) = \max(-d_i,0) + e_i ~~~\text{if }\infty \to i \text{ about } \alpha$$
     
Let $m_{i,\infty}(\alpha)$ (respectively $m_{\infty,i}(\alpha)$) be the number of edges distinguished on $\alpha$ on tile $i$ in $D$ where $i \to \infty$ about $\alpha$ (respectively $\infty \to i$ about $\alpha$.) Then for any boundary edge $\alpha$ on tile $i$,
     
     $$n_{i,\infty}(\alpha) = m_{i,\infty}(\alpha) ~~\text{ and }~~ n_{\infty, i}(\alpha) = m_{\infty,i}(\alpha).$$
\end{lemma}
\begin{proof}
The proof is similar to that of Lemma \ref{lemma:MijNij}.
\end{proof}
     
\begin{rmk}
Note that if we set $(d_\infty, e_\infty) = (0,0)$, our formula for the number of boundary edges in Lemma \ref{lemma:MinfinityN} coincides with our formula for the number of internal edges in Lemma \ref{lemma:MijNij} as 
     $$n_{i,\infty}(\alpha) = \max(d_i-d_\infty,0) + (e_\infty -e_i) ~~~\text{if }i \to \infty \text{ about } \alpha$$
     $$n_{\infty, i}(\alpha) = \max(d_\infty-d_i,0) + (e_i-e_\infty) ~~~\text{if }\infty \to i \text{ about } \alpha$$
     
\end{rmk}
     
\begin{cor}
Whenever $0 \leq \underline{e} \leq \underline{d}$, there are no antiedges on the boundary.
\end{cor}
     
We now prove Theorem \ref{thm:EtoD}.
     
\begin{proof}
Let $Q$ be a quiver of type $D_n$ where the associated quiver is acyclic. Let $\underline{d} \in \Phi_+$. Obtain the base graph $G(Q) = G$ and the minimal matching $M_-(Q, \underline{d}) = M_-$ via the procedures prescribed in Sections \ref{sec:base_graph} and \ref{sec:min_match}. Assign weights to the edges in $M_-$ as prescribed in Definition \ref{def:mixed_dimer}. Note that we agree to associate the vector $\underline{e} = (0,0,\dots,0)$, which trivially satisfies the all of Tran's conditions, to the minimal matching $M_-$ which is in $P$ by Proposition \ref{prop:minimalinP}.\allowbreak  \vspace{1em}
     
Suppose that we have some nonzero $\underline{e}$ satisfying all of Tran's conditions. To see that the resulting mixed dimer configuration $D$ after our performing our algorithm is reachable by a sequence of allowable flips, we apply Lemmas \ref{lemma:MijNij} and \ref{lemma:MinfinityN}. We proceed by induction on $|e|$. To establish the base case, we use $M_-$ which is by definition reachable by an empty sequence of allowable flips. Suppose that up to $k \in \nn$, our algorithm yields a mixed dimer configuration that is reachable by a sequence of allowable flips. Now suppose that $|e| = k+1$ and that we have added 1 to the $i^{\text{th}}$ entry of $\underline{e}$, i.e. flipped tile $i$. Note that as $i \to j$ is acceptable with respect to $(\underline{d}, \underline{e})$, we have that $e_i - e_j \leq \max(d_i-d_j,0)$,  i.e. $n_{i,j} \geq 0$. By Lemma \ref{lemma:MijNij}, we have that $n_{i,j} = m_{i,j}$ which implies that the number of edges straddling tiles $i$ and $j$ must be non-negative. Moreover, by Lemma \ref{lemma:MinfinityN}, since $0 \leq \underline{e} \leq \underline{d}$, we have that for any boundary edge $\alpha$ on tile $i$, $n_{i,\infty}(\alpha) = \max(d_i,0) - e_i = d_i - e_i \geq 0$ giving that $m_{i,\infty}(\alpha)$, the number of edges on $\alpha$ in $D$ is non-negative. Similarly, $n_{\infty,i}(\alpha) = \max(-d_i,0) + e_i = 0+e_i \geq 0$ giving that $m_{\infty,i}(\alpha)$, the number of edges on $\alpha$ in $D$ is non-negative. Hence, the resulting mixed dimer configuration $D$ has edges with all non-negative weights. With this, we claim that there exists some reading of the entries of $\underline{e}$ (not simply left to right as the algorithm prescribes) which yields a mixed dimer configuration with non-negative weights at each flip along the way in the sequence. \allowbreak  \vspace{1em}

Let $j$ denote the label of a vertex of the quiver $Q$ such that $e_j>0$ and such that for any arrow pointing $i \to j$, we have $e_j > e_i$.  Note that since $|e|  = k+1 > 0$ and $Q$ is an oriented tree, it follows that such a vertex $j$ exists.  Consequently, if we let $\underline{e'}$ be the result of subtracting the $j^{\text{th}}$ unit vector from $\underline{e}$, we see that (1) for every arrow $i \to j$, we have  
     $n_{i,j}' = \max(d_i-d_j,0) + (e_j-1-e_i)$ is nonnegative; (2) for every arrow $j \to i$, we have 
     $n_{j,i}' = \max(d_j-d_i,0) + (e_i - e_j+1)$ is nonnegative; (3) the value $n_{i, \infty}' = d_i - e_i$ is nonnegative, and (4) the value $n_{\infty, i}' = e_i$ is nonnegative.\allowbreak  \vspace{1em}
     
Hence, utilizing Lemmas \ref{lemma:MijNij} and \ref{lemma:MinfinityN}, we see that $e'$ corresponds to a mixed dimer configuration $D'$.  Since $|e'|=k$, by the inductive hypothesis, $D'$ is reachable by a sequence of allowable flips from $M_-$.  Hence $D$ itself is reachable from $M_-$ by a sequence of allowable flips, where we tack on a flip of tile $j$.  This last flip is allowable since both $D'$ and $D$ are mixed dimer configurations, i.e. with nonnegative weights on edges, and the difference $\underline{e} - \underline{e'}$ is the $j^{\text{th}}$ unit vector.\allowbreak  \vspace{1em}
     
Hence, any mixed dimer configuration with non-negative weights must be reachable via a sequence of allowable flips from $M_-$. Therefore, since the resulting mixed dimer configuration $D$ from our algorithm has all non-negative weights, it is reachable via a sequence of allowable flips from $M_-$.\allowbreak  \vspace{1em}
     
The last thing we need to show is that the resulting mixed dimer configuration $D$ is node-monochromatic. We again proceed by induction on $|e|$. By Proposition \ref{prop:minimalinP}, we have that $M_-$ is node-monochromatic, establishing the base case. So suppose that up to $k \in \nn$, we have that when $|e| = k$, the resulting mixed dimer configuration $D$ is node-monochromatic. Now take $\underline{e}$ with $|e| = k+1$ where we have added 1 to the $i^{\text{th}}$ entry of $\underline{e}$. Since we have previously established that this mixed dimer configuration $D$ is reachable by an allowable sequence of flips, we can assume that $D$ was obtained via a sequence of $k+1$ allowable flips from $M_-$. We claim that the only way we could have produced a path connecting nodes of different colors is if there is more than one critical arrow in $Q$ with respect to $(\underline{d}, \underline{e})$. We first show that if a path between \textcolor{red}{$u,v$} and \textcolor{blue}{$w,x$} can be formed, then the corresponding quiver must have at least two critical arrows. We start with the minimal matchings for each orientation of the last three vertices $n-3, n-2$ and $n-1$ in $Q$ and show that the only way to obtain paths between blue and red nodes via a sequence of allowable flips is when two critical arrows (namely, violating Tran's condition (3)) involving these three vertices arise:

     \begin{center}
         \includegraphics[scale=.2]{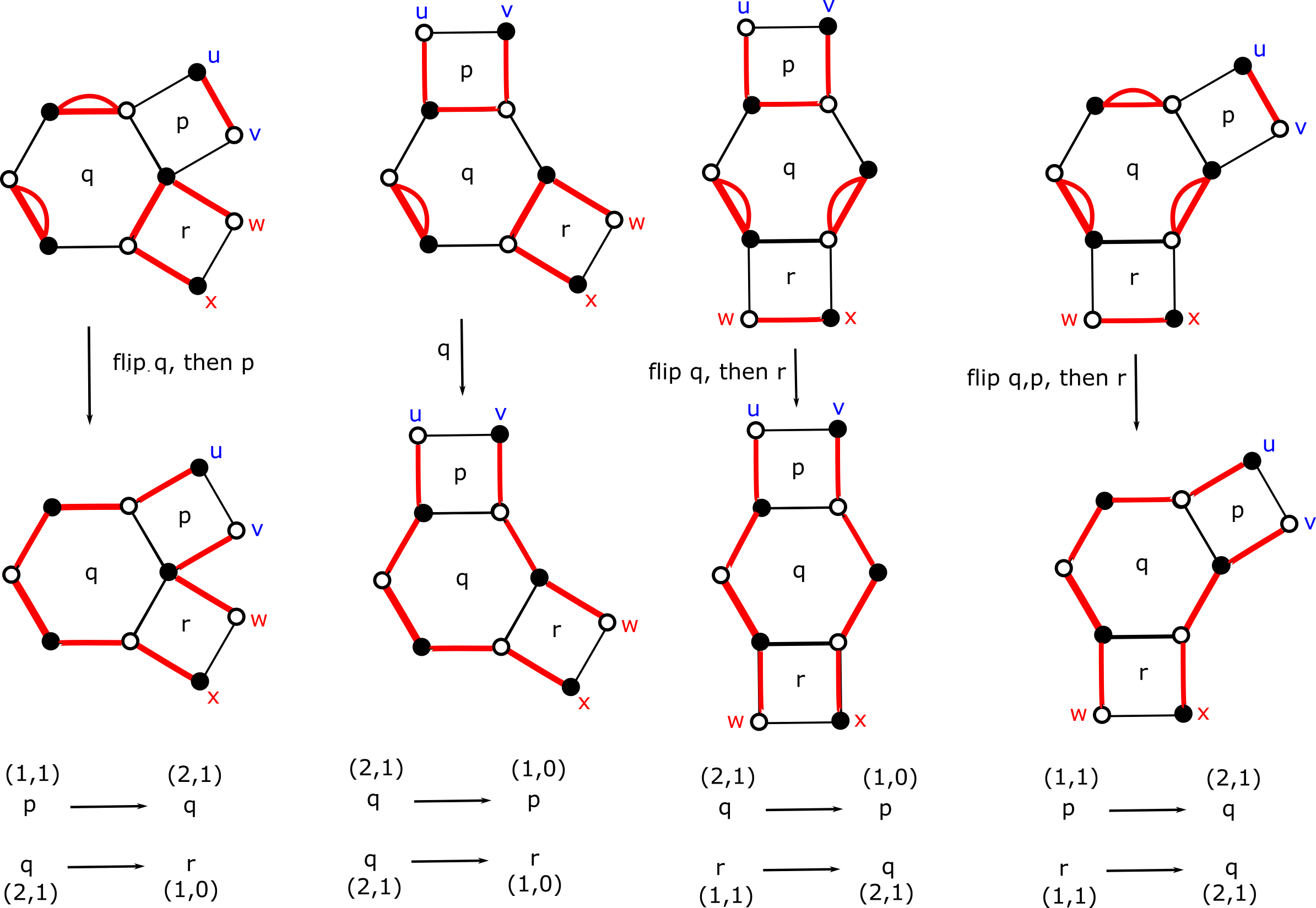}
     \end{center}
     
By the symmetry of the vertices $n-1$ and $n-2$ in $Q$, it suffices to show no paths between \textcolor{olive}{green} nodes to either \textcolor{blue}{blue} or \textcolor{red}{red} can occur. To do this, we show that if such a path exists, there would be two critical arrows: one being the arrow $i \to i-1$ where $d_i$ is the 2 in the $\underline{d}$-vector of maximal index and one involving the vertices $n-3$ and one of $n-2$ or $n-1$ (connecting to \textcolor{red}{$u,v$} or \textcolor{blue}{$w,x$} respectively). To simplify the case work, suppose that one of the arrows $n-3 \to n-2$ and $n-3 \to n-1$ are critical. With this assumption, we can potentially make paths connecting both \textcolor{red}{red} to \textcolor{olive}{green} and/or \textcolor{blue}{blue} to \textcolor{olive}{green}. We show that is this case, an additional critical arrow must be created in order to connect \textcolor{red}{red} or \textcolor{olive}{green} to \textcolor{blue}{blue}. Suppose that $i$ is the 2 of maximal index in the $\underline{d}$ vector. For the definition of where to place the nodes $y,z$, we need to examine the cases where $i, i-1, i-2$ are zigzagging or straight. \allowbreak \vspace{1em}

Note that if $i-2 <0$, then this case reduces to the straight case. Suppose that the tiles $i, i-1, i-2$ are placed straight in $G$. Then, in order to create a path from either \textcolor{red}{red} or \textcolor{olive}{green} to \textcolor{blue}{blue}, we must flip all tiles with $d_j = 2$ for $j \geq i$ to obtain edges that are from white to black clockwise with respect to $G_2$. This gives that $(d_j,e_j) = (2,1)$ for all vertices $j \geq i$. When we flip tile $i$, note that we will connect the \textcolor{olive}{green} nodes to both the \textcolor{red}{red} and \textcolor{blue}{blue} nodes as pictured below in the left column:

     \begin{center}
         \includegraphics[scale=.2]{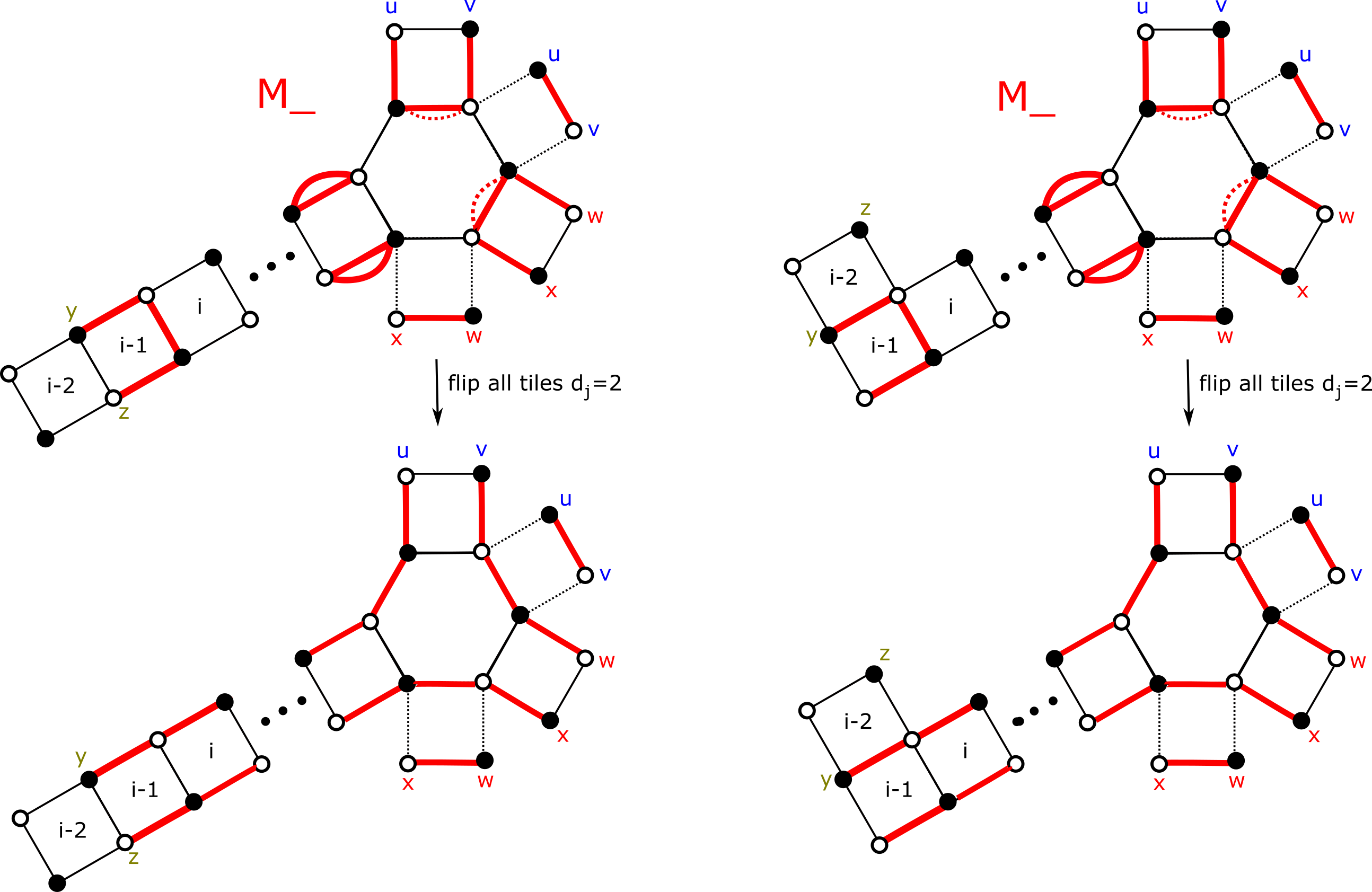}
     \end{center}
     
We see that this must create a critical arrow $i \to i-1$ as $(d_i,e_i) = (2,1)$ and $(d_{i-1},e_{i-1}) = (1,0)$. Now consider when $i, i-1, i-2$ are stacked in a zigzagging fashion in $G$. As before, in order to create a path from either \textcolor{red}{red} or \textcolor{olive}{green} to \textcolor{blue}{blue}, we must flip all tiles with $d_j = 2$ for $j \geq i$ to obtain edges that are from white to black clockwise with respect to $G_2$. This gives that $(d_j,e_j) = (2,1)$ for all vertices $j \geq i$. When we flip tile $i$, note that we will connect the \textcolor{olive}{green} node $y$ to the \textcolor{blue}{blue} node $u$ as pictured in the above figure on the right.\allowbreak  \vspace{1em}

Thus, we see that via a sequence of allowable flips from $M_-$, this mixed dimer configuration will remain node-monochromatic as $\underline{e}$ satisfies Tran's criticality condition. Hence, we have associated a mixed dimer configuration $D$ in our poset $P$ to our given vector $\underline{e}$ satisfying all of Tran's conditions. 
     \end{proof}
     
To continue our proof of Theorem \ref{thm:biject}, we now construct a map $D \longrightarrow \underline{e}$. The map sending $D \longrightarrow \underline{e}$ does not require the weighted version of flipping used in the previous theorem. For this direction of the bijection, we restrict our attention to allowable flips. This map is given by the superimposition of $D$ with $M_-$ and it is constructed via the following theorem:

\begin{thm}
\label{thm:DtoE}
Given $(Q,\underline{d}, D)$, there exists a unique way to produce $\underline{e}$ that satisfies all of Tran's conditions i.e. conditions (1), (2), and (3) of Theorem \ref{thm:Tran}.  The process is given by the following procedure: 
     
     \begin{enumerate}
         \item Let $\underline{v} = (0,0, \dots, 0) \in \zz^n$. Given $(Q,\underline{d}, D)$, construct the corresponding base graph and minimal matching $M_-$ via the procedure described in Section \ref{sec:base_graph}. 
         \item Superimpose $D$ with $M_-$ to obtain the multigraph $M_1 := D \sqcup M_-$.
         \item Using the edges distinguished in $D$ and $M_-$, create the largest cycle of length $\ell_1 > 2$, call it $C_1$. (Note that $C_1$ may not be the unique cycle of length $\ell_1$). For all faces $i$ enclosed by $C_1$, add $+1$ to $v_i$ in $\underline{v}$ and delete $C_1$ from $M_1$. Call the resulting multigraph $M_2 = M_1 \setminus C_1$.
         \item Examine $M_2$ and if there are any cycles of length $\ell_2 > 2$, find the cycle of largest length and call it $C_2$. For all faces $i$ enclosed by $C_2$, add +1 to $v_i$ and delete $C_2$ from $M_2$.
         \item Iterate this process of deleting cycles of of largest length and adding 1's to the vector $\underline{v}$ until nothing is left besides 2-cycles. The resultant vector $\underline{v}$ corresponds to $\underline{e}$.
     \end{enumerate}
\end{thm}

Before proving this theorem, we compute an example:
 
 \begin{ex}
       \label{ex:Dtoe}
  Consider the quiver $Q$ from Example \ref{ex:etoD} with the same positive root $\underline{d} = (1,1,2,2,1,1)$. Suppose we start with the mixed dimer configuration $D \in P$ given by 
  
  \begin{center}
      \includegraphics[scale=.25]{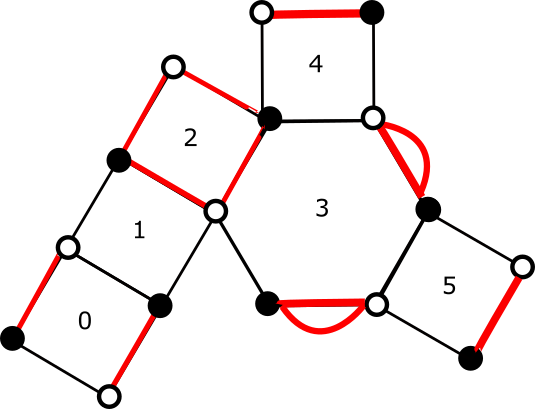}
  \end{center}
  
  \noindent which is obtained by the sequence of allowable flips at tiles 3,2,4,5,3 in order. Superimposing this with the minimal matching, distinguished in blue, we obtain the multigraph as shown here.
  
    \begin{center}
      \includegraphics[scale=.25]{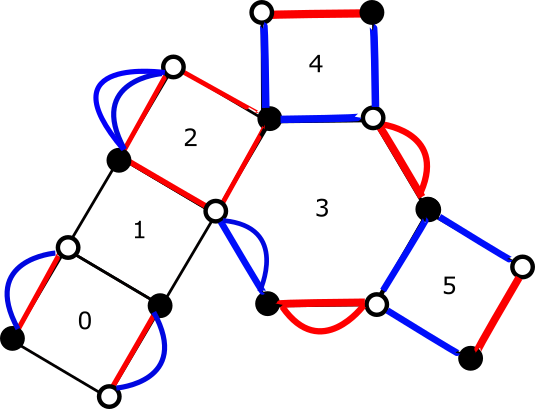}
  \end{center}
  
  Following the algorithm stated in Theorem \ref{thm:DtoE}, we see the largest cycle we can create is of length 12 enclosing tiles 2,3,4 and 5, call it $C_1$ which we distinguish in green.
  
    \begin{center}
      \includegraphics[scale=.25]{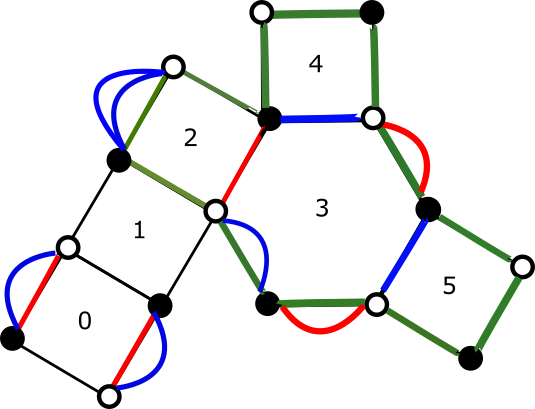}
  \end{center}
  
  Deleting $C_1$ from our multigraph yields the resulting configuration as illustrated,
  
    \begin{center}
      \includegraphics[scale=.25]{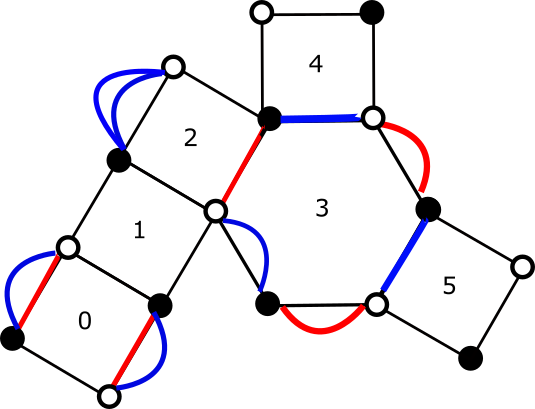}
  \end{center}
  and since this cycle enclosed tiles 2,3,4 and 5, we now have that $\underline{v} = (0,0,1,1,1,1)$. Looking for the next largest cycle, we see that the cycle enclosing the hexagon i.e. tile 3 is largest, call it $C_2$. Deleting $C_2$ from our multigraph, we obtain the figure below,
  
    \begin{center}
      \includegraphics[scale=.25]{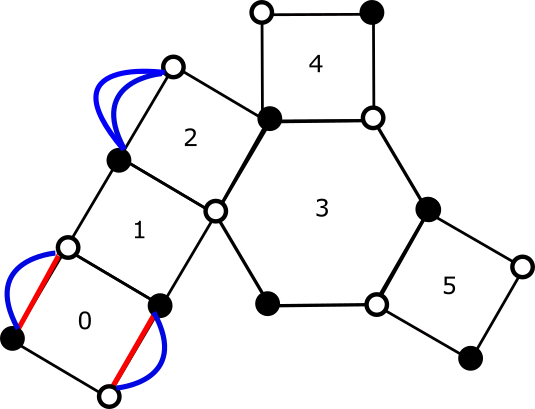}
  \end{center}
  and we add $+1$ to $v_3$, giving that $\underline{v} = (0,0,1,2,1,1)$. Because there are only 2-cycles left in our multigraph, we have that $\underline{v} = (0,0,1,2,1,1)$ is the $\underline{e}$ we associate to $D$.
 \end{ex}
 
 We now prove Theorem \ref{thm:DtoE}.
     \begin{proof}
     Let $Q$ be a quiver of type $D_n$ where the associated quiver is acyclic. Let $\underline{d} \in \Phi_+$. Obtain the base graph $G(Q) = G$
     and $M_-(Q,\underline{d}) = M_-$ via the procedure prescribed in Sections \ref{sec:base_graph} and \ref{sec:min_match}.
     Let $D \neq M_-$ be a mixed dimer configuration in $P$ and initialize $\underline{v} = (0,0,\dots,0)$.\allowbreak  \vspace{1em}

     We first illustrate why we can delete a cycle from the superimposition of $D$ and $M_-$. Note that as both $D, M_- \in P$, they individually satisfy the valence condition. Hence, we obtain that each vertex that was previously valence 2 is now valence 4 in $M_1$ and each vertex that was valence 1 is now valence 2 in $M_1$. Because $D \neq M_-$, there must exist some cycle of length at least 4 in their superimposition as $D$ differs from the minimal matching by at least one allowable flip. Hence, there exists some cycle of maximal length $\ell_1 \geq 4$. \allowbreak  \vspace{1em}
     
     We aim to show that after each iteration of cycle deletion and adding the corresponding 1's to the entries in the vector $\underline{v}$, the resultant vector satisfies all of Tran's conditions. \allowbreak \vspace{1em}
     
     \textbf{Tran Condition (1)}. We claim that Condition (1) of Theorem \ref{thm:Tran}, i.e. $0 \leq \underline{v} \leq \underline{d}$ is satisfied. Note that the only way a cycle can enclose tile $j$ in the superimposition of $D$ and $M_-$ is if $D$ differs from $M_-$ at a tile $j$. The only way this could happen is if an allowable flip occurred at tile $j$. By definition of allowable flip, we require that $d_j > 0$. Moreover, the only positive entries in $\underline{v}$ correspond to tiles enclosed by some cycle $C_i$ in our algorithm. The only way $v_j > d_j$ is if $v_j = 2$ and $d_j = 1$. However, this would imply that the tile $j$ was enclosed by two cycles in our algorithm i.e. the tile $j$ can be flipped twice. This can only happen in $d_j = 2$. Hence, $v_j \leq d_j$ for all $j$ giving that $0 \leq \underline{v} \leq \underline{d}$.\allowbreak  \vspace{1em}
      
      \textbf{Tran Condition (2).} In order to prove the other two Tran conditions, we will be using two forms of induction. We will be inducting of the number of iterations needed such that nothing is left in $D \sqcup M_-$ but a disjoint union of 2-cycles. Within this induction, we will also induct on the number of tiles that each cycle in the algorithm encloses. We can establish the base case of both inductions immediately as we agree to associate $M_-$ to $\underline{e} = (0,0,\dots,0)$ which trivially satisfies Tran's conditions (2). We first claim that deleting $C_i$ and adding the appropriate 1's to $\underline{v}'$, we obtain a vector $\underline{v}$ that satisfies Tran's acceptability condition, i.e. Condition (2) of Theorem \ref{thm:Tran}. In order to do this, we now induct on the number of tiles enclosed by $C_i$. \allowbreak  \vspace{1em}
     
     Suppose that $C_i$ encloses one face, call it $j$. In order to check that $\underline{v}$ is acceptable in this case, it suffices to check that any arrows involving vertex $j$ satisfies the acceptability condition. Note that $v_{j} = 1$ or $2$ because $j$ could have been enclosed in a previous cycle in an earlier iteration. We first address the case when $v_{j}=1$. Also, suppose that $j$ is the tail of an arrow $k \to j$. In this case, when $v_{k} \neq 2$, the acceptability is trivially satisfied as $v_{k} - v_{j} \leq 0$. So, we consider when $v_{k} =2$ which gives us that $v_{k} - v_{j} = 1$. Note that all the cases where acceptability fails, i.e. when $\max(d_k-d_j,0) = 0$, are precisely the cases where 
     
     $$n_{k,j} = \max(d_k-d_j,0) + (v_{j} - v_{k}) = 0 -1 = -1$$
     
      However, this violates Corollary \ref{cor:accnonneg} which tells us that all edges weights must be nonnegative. Hence, these cases are impossible. \allowbreak  \vspace{1em}
      
      Now suppose that $j$ is the head of an arrow, i.e. $j \to k$ in $Q$.  In this case, when $v_{k} \neq 0$, the acceptability is trivially satisfied as $v_{k} - v_{j} \leq 0$. So, we consider when $v_{k} =0$ which gives us that $v_{j} - v_{k} = 1$. Note that all the cases where acceptability fails, i.e. when $\max(d_j-d_k,0) = 0$, are precisely the cases where 
     
     $$n_{j,k} = \max(d_j-d_k,0) + (v_{k} - v_{j}) = 0 -1 = -1$$
     
      Again, this violates Corollary \ref{cor:accnonneg} as all edge weights must be nonnegative. Hence, this cases is also impossible. \allowbreak  \vspace{1em}

      Now suppose that $v_{j} = 2$. Note that $j$ is the tail of an arrow $k \to j$, the acceptability is trivially satisfied as $v_{k} - v_{j} \leq 0$ for any choice of $v_{k}$. So, it suffices to consider when $j$ is the head of an arrow, i.e. $j \to k$ in $Q$.  In this case, when $v_{k} = 2$, the acceptability is trivially satisfied as $v_{j} - v_{k} = 0$. So, we consider when $v_{k} =0$ or 1. If $v_{k} = 1$, then $v_{j} - v_{k} = 1$ and all the cases where acceptability fails, i.e. when $\max(d_j-d_k,0) = 0$, are precisely the cases where 
     
     $$n_{j,k} = \max(d_j-d_k,0) + (v_{k} - v_{j}) = 0 -1 = -1$$
     
      However, this again violates Corollary \ref{cor:accnonneg}. Similarly, if $v_{k} = 0$, then $v_{j} - v_{k} = 2$. Then, all the cases where acceptability fails, i.e. when $\max(d_k-d_j,0) \leq 1$, are precisely the cases where 
     
     $$n_{j,k} = \max(d_j-d_k,0) + (v_{k} - v_{j}) < 0$$
      violating Corollary \ref{cor:accnonneg}. Hence, these cases are impossible. Therefore, we have established the base case for our induction on the number of tiles enclosed by $C_i$.\allowbreak  \vspace{1em}
      
      Now, assume that if $C_i$ enclosed $k$ tiles, then $\underline{v}$ satisfies Tran's acceptability condition. Now, suppose that $C_i$ encloses $k+1$ tiles, $t_1, \dots, t_{k+1}$. We aim to show that $\underline{v} + \underline{e_{t_{k+1}}} = \underline{y}$ satisfies Tran's acceptability condition. \allowbreak  \vspace{1em}
      
      Again, by our inductive hypothesis, it suffices to show that any arrow with vertex $t_{k+1}$ still satisfies the acceptability condition. Since tile $t_{k+1}$ is enclosed by $C_i$ and may have been enclosed by another cycle in a previous iteration, we have that $y_{t_{k+1}} \geq 1$. If $j$ is the other vertex in an arrow involving $t_{k+1}$, we also have to consider all possible values of $d_j,y_j, d_{t_{k+1}}$ as well as if $t_{k+1}$ is the head or tail of the arrow. First, suppose that $y_{t_{k+1}} = 1$ which means that $d_{t_{k+1}} = 1$ or 2. Next, suppose that $t_{k+1}$ is the tail of an arrow, i.e. $j \to t_{k+1}$ in $Q$. When $y_j < 2$, the acceptability condition is trivially satisfied as $y_j - y_{t_{k+1}} \leq 0$. So suppose that $y_j = 2$ and since $y_j \leq d_j$, this implies that $d_j = 2$. As $d_j \geq 1$, the acceptability condition is satisfied in all cases. \allowbreak  \vspace{1em}
      
      Now, still suppose that $y_{t_{k+1}} = 1$, but now suppose that $t_{k+1}$ is the head of an arrow, i.e. $t_{k+1} \to j$ in $Q$. When $y_j > 0$, the acceptability condition is trivially satisfied as $y_{t_{k+1}} - y_j \leq 0$. So suppose that $y_j = 0$, note that if $d_{t_{k+1}} - d_j \geq 1$, the acceptability condition is satisfied. When $d_{t_{k+1}} - d_j \leq 0$, this means that  
      
      $$n_{t_{k+1},j} = \max(d_{t_{k+1}}-d_j,0) + (y_j-y_{t_{k+1}}) = -1$$
      violating Corollary \ref{cor:accnonneg}. Hence, these cases are impossible. Now, suppose that $y_{t_{k+1}} = 2$ which means that $d_{t_{k+1}} = 2$. Note that when $t_{k+1}$ is the tail of an arrow, i.e. $j \to t_{k+1}$ in $Q$, $y_j-y_{t_{k+1}} \leq 0$ giving that acceptability is trivially satisfied. So suppose that $t_{k+1}$ is the head of an arrow, i.e. $t_{k+1} \to j$ in $Q$. When $y_j = 2$, the acceptability condition is trivially satisfied as $y_{t_{k+1}} - y_j = 0$. So suppose that $y_j = 1$ or 2, note that as $d_{t_{k+1}} = 2$, the acceptability condition only fails if $y_j < d_j$. But this is precisely when
      
      $$n_{t_{k+1},j} = \max(d_{t_{k+1}}-d_j,0) + (y_j-y_{t_{k+1}}) = -1$$
      violating Corollary \ref{cor:accnonneg}. Hence, these cases are impossible. Therefore, we have shown that if $C_i$ encloses $k+1$ cycles, the resulting vector $\underline{y}$ satisfies Tran's acceptability condition. By induction, we have that after the $i^{\text{th}}$ step of our algorithm, $\underline{v}$ satisfies Tran's acceptability condition. Therefore, any vector obtained via this algorithm must satisfy the acceptability condition. Hence, the resultant $\underline{e}$ must be acceptable. \allowbreak  \vspace{1em}
     
     \textbf{Tran Condition (3).} We now show that the criticality condition, Tran's condition (3), is satisfied by $\underline{v}$. In order to do this, we again induct on the number of steps needed to complete our algorithm. Furthermore, we induct on the number of tiles enclosed by $C_i$ at each step. \allowbreak  \vspace{1em}
     
     To establish the base cases of both inductions, suppose that $C_1$ encloses a unique tile $j$. Then $\underline{v} = \underline{e_j}$, the unit vector with a 1 in the $j^{\text{th}}$ position. Note that if $d_j \neq 2$, no critical arrows can be formed as $S := \{i \in Q ~:~ (d_i,e_i) = (2,1)\} = \emptyset$. So, suppose that $d_j = 2$. Then, as $(d_j,v_{j}) = (2,1)$, $j \in S$ and the only way that the criticality condition could have failed is if this created two critical arrows i.e.
     
     $$k \longleftarrow j \longrightarrow \ell$$
     $$(1,0) \leftarrow (2,1) \rightarrow (1,0)$$
    where we must have $v_{k}=0=v_{\ell}$ as the only tile enclosed by a cycle is $j$. Note that since the subgraph $G_2$ associated to all tiles with $d$ entry 2 is connected, this means that $j$ must be the unique tile in $G_2$. By the structure of the positive root $\underline{d} \in \Phi_+$, we have that $j$ must be on the tile associated to the hexagon, i.e. vertex $n-3$ in $Q$.  Hence, our base graph is:\allowbreak  \vspace{1em}

    \begin{figure}[H]
    \centering 
        \includegraphics[scale=.25]{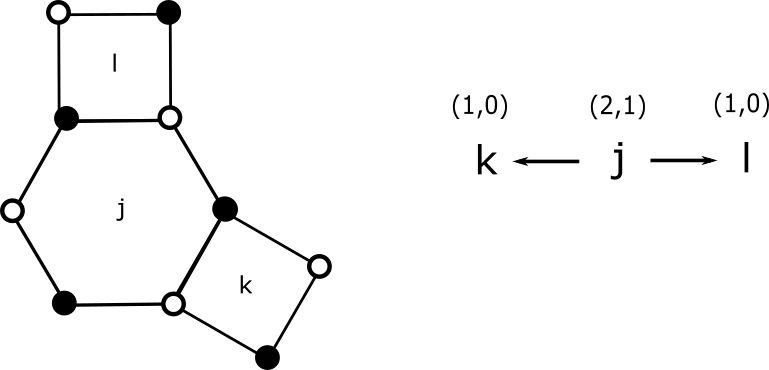}
        \caption{ }
        \label{fig:fig1}
    \end{figure}
    \noindent based on the orientation of the quiver. In this case, our minimal matching and mixed dimer configuration $D$ will locally look like the following:
    
    \begin{figure}[H]
     \centering
        \includegraphics[scale=.25]{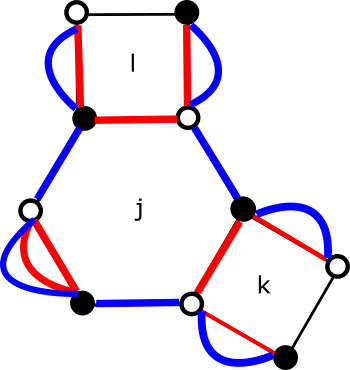}
        \caption{ }
        \label{fig:fig2}
    \end{figure}
    
    \noindent where we obtained the minimal matching, distinguished in \textcolor{red}{red}, with the convention that we distinguish  edges going black to white in the bipartite coloring and obtain $D$, distinguished in \textcolor{blue}{blue}, by performing an allowable flip at tile $j$. We now have two cases: either both $k, \ell$ are terminal vertices, i.e. are degree one in the quiver, or one of $k$ or $\ell$ is terminal and the other is not, meaning it has degree two in the quiver. Suppose that both $k, \ell$ are terminal. With the labeling of the nodes as prescribed in Section \ref{sec:nodes}, we see that the mixed dimer configuration $D$ has two paths connecting nodes of different colors.
    
    \begin{figure}[H]
    \centering
        \includegraphics[scale=.25]{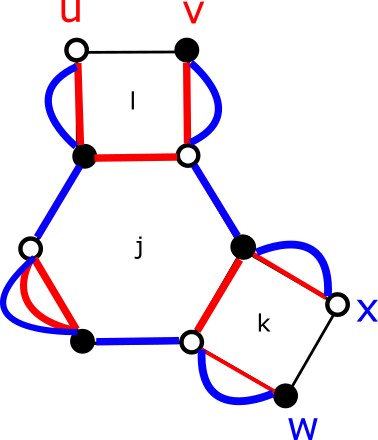}
        \caption{ }
        \label{fig:fig3}
    \end{figure}
    
    This means that $D \not\in P$ which is a contradiction. Now, without loss of generality, suppose that $k$ is the terminal vertex and $\ell$ has degree two in the quiver. Since $j$ is the minimal index with $d_j = 2$ in $\underline{d}$, the nodes $y,z$ will either be placed on tile $\ell = j-1$ or $j-2$ depending on the orientation of the arrows connecting $j, \ell, j-2$. In either case, we see paths between nodes of different colors formed:
    
    \begin{figure}[H]
    \centering
        \includegraphics[scale=.25]{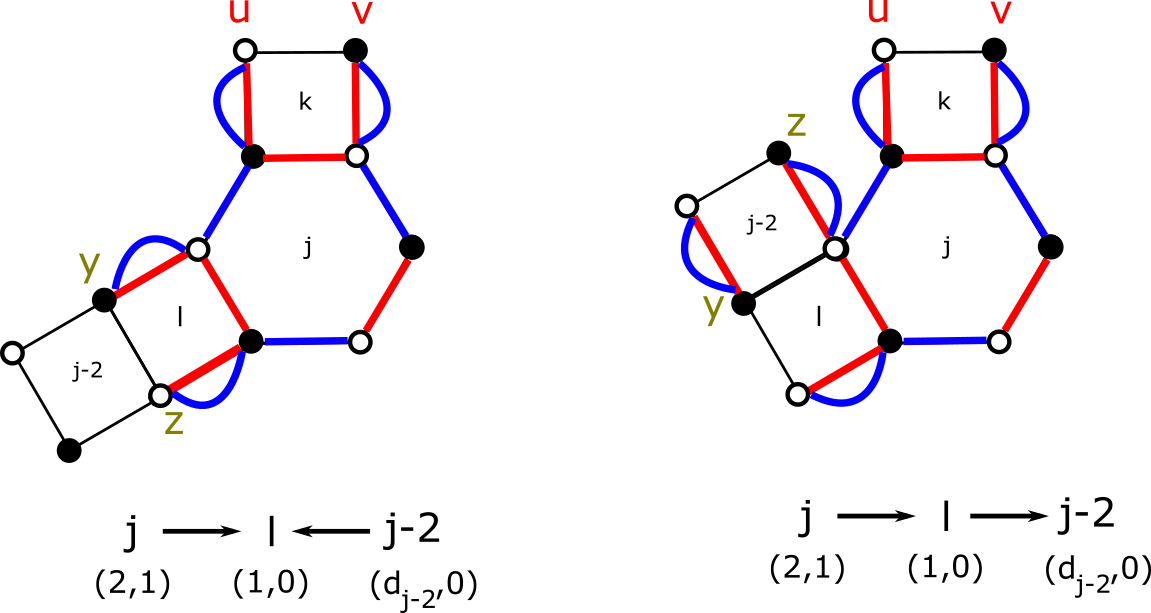}
        \caption{ }
        \label{fig:fig4}
    \end{figure}
    
    Hence, two critical arrows cannot be formed if $C_1$ encloses a single tile.  Now, suppose that if $C_1$ encloses $k$ tiles, the resulting vector $\underline{v}$ satisfies the criticality condition. We now aim to show that if $C_1$ encloses $k+1$ tiles, then the resulting vector $\underline{w} := \underline{v} + \underline{e_j}$ satisfies the criticality condition. Note that by our inductive assumption, $(\underline{v}, \underline{d})$ has at most one critical arrow. Suppose that $(\underline{v}, \underline{d})$ created 0 critical arrows. Since $\underline{w}$ only differs from $\underline{v}$ in the $j^{\text{th}}$ entry, it suffices to only check any arrows involving $j$. Moreover, we need to show that two critical arrows cannot be created by adding $+1$ to the $j^{\text{th}}$ entry of $\underline{v}$ to obtain $\underline{w}$. The only way that this could happen is if $j$ is the unique vertex in $S$, i.e. $d_j$ is the unique 2 in $\underline{d}$ with a quiver of one of the following orientations:
    
        \begin{align*}
             p~~ \longleftarrow ~~&j~~ \longleftarrow ~~q \tag{i}\\ 
            (1,0) \leftarrow (2&,1) \leftarrow (1,1)\\
             p~~ \longrightarrow ~~&j~~ \longrightarrow ~~q \tag{ii}\\
            (1,1) \rightarrow (2&,1) \rightarrow (1,0)\\
             p~~ \longleftarrow ~~&j~~ \longrightarrow ~~q \tag{iii}\\
            (1,0) \leftarrow (2&,1) \rightarrow (1,0)\\
             p~~ \longrightarrow ~~&j~~ \longleftarrow ~~q \tag{iv}\\
            (1,1) \rightarrow (2&,1) \leftarrow (1,1)
        \end{align*}
    Note that case (iv) will not be possible as if $w_p = w_q = 1$. This means that both $p$ and $q$ were enclosed by $C_1$. However, since $j$ is in between these vertices, this impossible without having enclosed $j$ as well by the connectedness of the tiles enclosed by $C_1$.\allowbreak  \vspace{1em}
    
    Note that we analyzed case (iii) in the base case. Namely, since $C_1$ enclosed a connected set of tiles, we must have that $p,q$ are terminal which puts us in the case of Figure 4.\allowbreak  \vspace{1em}
    
    Also note that cases (i), (ii) are symmetric. So it suffices to focus on case (i). Based on the orientations of the arrows, our base graph is up to isometry/ bipartite coloring is one of the following:\allowbreak  \vspace{1em}
    
    \begin{figure}[H]
    \centering
        \includegraphics[scale=.25]{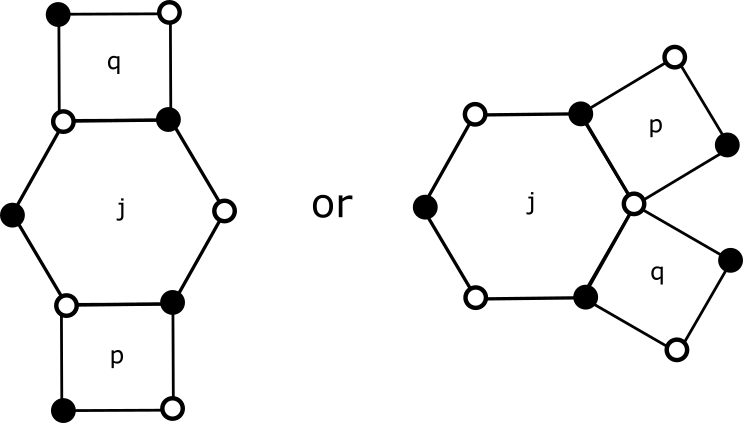}
        \caption{ }
        \label{fig:fig5}
    \end{figure}
    
    Within these cases, we have to consider if $p,q$ are terminal vertices in the quiver. This gives that for each of these base graphs, there are two cases: either both $p,q$ are terminal in $Q$ or one of $p$ or $q$ is terminal in $Q$. Suppose that the base graph is given by the left graph in Figure 6. Then, if both $p,q$ are terminal, we obtain the mixed dimer configuration $D$ by performing allowable flips at tile $j$, then tile $q$. In this case, we see $D$ is node-polychromatic as:
    
     \begin{figure}[H]
     \centering
        \includegraphics[scale=.25]{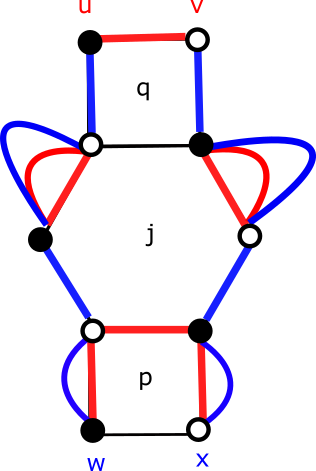}
        \caption{ }
        \label{fig:fig6}
    \end{figure}
    
    Without loss of generality, suppose that $q$ is terminal and $p$ has degree 2 in $Q$. Namely, in either case of the orientation of this arrow, we see $D$ is node-polychromatic as:\allowbreak  \vspace{1em}
    
    \begin{figure}[H]
    \centering
        \includegraphics[scale=.25]{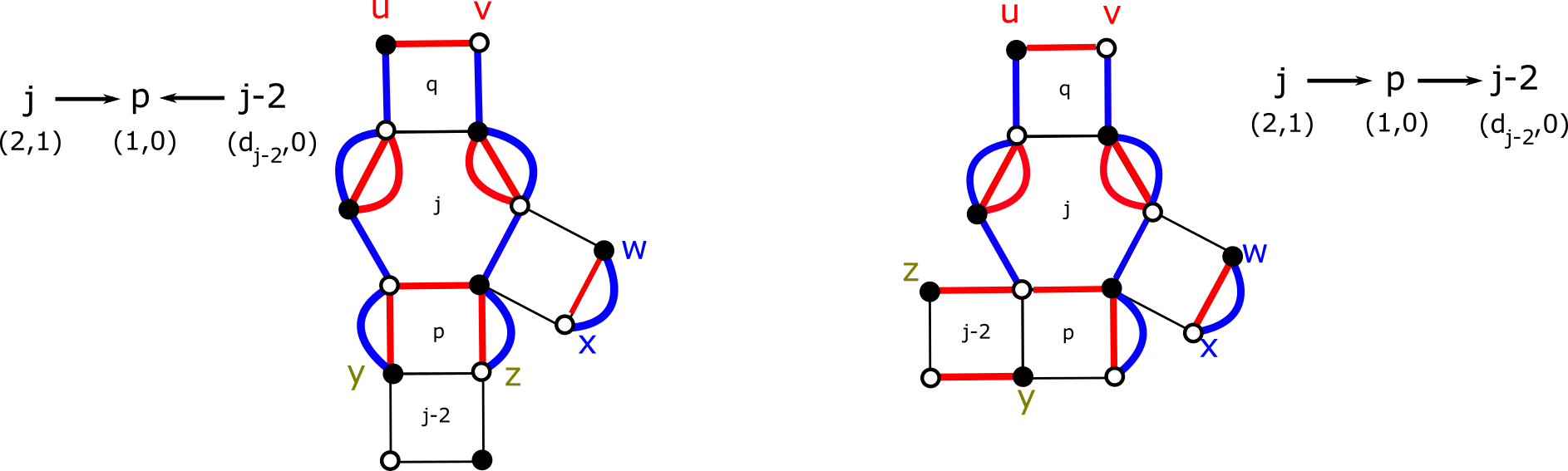}
        \caption{ }
        \label{fig:fig7}
    \end{figure}
    
    Now, suppose the base graph is given by the right graph in Figure 6 and suppose that both $p,q$ are terminal. Then, we see $D$ is node-polychromatic as:
    
    \begin{figure}[H]
    \centering
        \includegraphics[scale=.25]{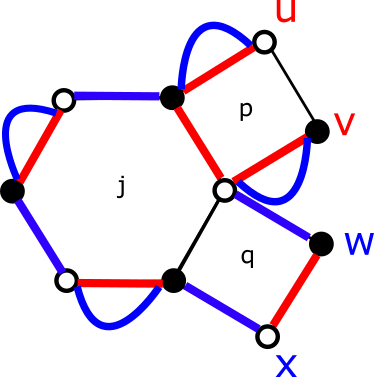}
        \caption{ }
        \label{fig:fig8}
    \end{figure}
    
    Without loss of generality, suppose that $q$ is terminal and $p$ has degree 2 in $Q$. Namely, in either case of the orientation of this arrow, we analyzed isometric copies of these cases in Figure 8 and saw they produce node-polychromatic mixed dimer configurations. \allowbreak  \vspace{1em}
    
    So, if no critical arrows were created when $C_1$ enclosed $k$ tiles, we have shown that two critical arrows are not created when $C_1$ encloses $k+1$ tiles. Now suppose that one critical arrow is created if $C_1$ encloses $k$ tiles. We now aim to show that enclosing $k+1$ tiles creates no critical arrows. Note that $S \neq \emptyset$ as one critical arrow was created. Suppose that $s \in S$, i.e. $(d_s, w_s) = (2,1)$, and as $v_s = 1$, we know that $s$ must have been enclosed by $C_1$. Note that if $j$ were not connected to $s$, then the only way another critical arrow can be made is if $(d_j,v_j) = (2,1)$. However, entries that are 2 in $\underline{d} \in \Phi_+$ are consecutive, so this cannot happen. So suppose that $j$ is connected to $s$. The only way that another critical arrow can be made is if either $j \to s$ with $(d_j,w_j) = (1,1)$ and $(d_s,w_s) = (2,1)$ or $j \leftarrow s$ with $(d_j,w_j) = (1,0)$ and $(d_s,w_s) = (2,1)$. However, note that the ladder case is impossible as $w_j=1$ by hypothesis that $j$ is enclosed by $C_1$. So, now we either have 
    
     \begin{align*}
             j~~ \longrightarrow ~~&s~~ \longleftarrow ~~t \tag{a}\\
            (1,1) \rightarrow (2&,1) \leftarrow (1,1)\\
             j~~ \longleftarrow ~~&s~~ \longrightarrow ~~t \tag{b}\\
            (1,0) \leftarrow (2&,1) \rightarrow (1,0)
        \end{align*}

        Note that case (b) was done previously where $j$ plays the role of $q$ and $s$ plays the role of $j$ in case (i) where we assumed that zero critical arrows were formed with $C_1$ enclosing $k$ tiles. So, we only need to address case (a). Again, $s \in S$ must be the unique vertex with $d_s = 2$ as the 2's are consecutive in any positive root which gives us that $s$ must be the hexagon. Note that with the orientation of the arrows, the only possibility for the base graph up to isometry is:
        
    \begin{figure}[H]
    \centering
        \includegraphics[scale=.25]{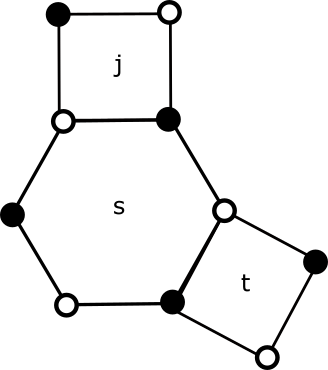}
        \caption{ }
        \label{fig:fig9}
    \end{figure}
        
        Now we consider if $j,t$ are terminal vertices. If $j,t$ are both terminal, note that a node-polychromatic mixed dimer configuration is produced as:\allowbreak  \vspace{1em}
        
    \begin{figure}[H]
    \centering
        \includegraphics[scale=.25]{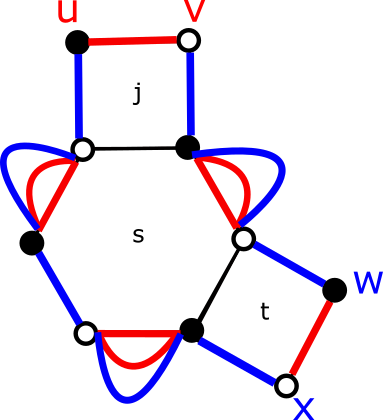}
        \caption{ }
        \label{fig:fig10}
    \end{figure}
        
        Without loss of generality, suppose that $j$ is terminal and $t$ has degree 2 in $Q$. Then, we need to consider the other arrow with vertex $t$ to use the nodes $y,z$. In these cases, we obtain paths between nodes of different colors:
        
    \begin{figure}[H]
    \centering
        \includegraphics[scale=.25]{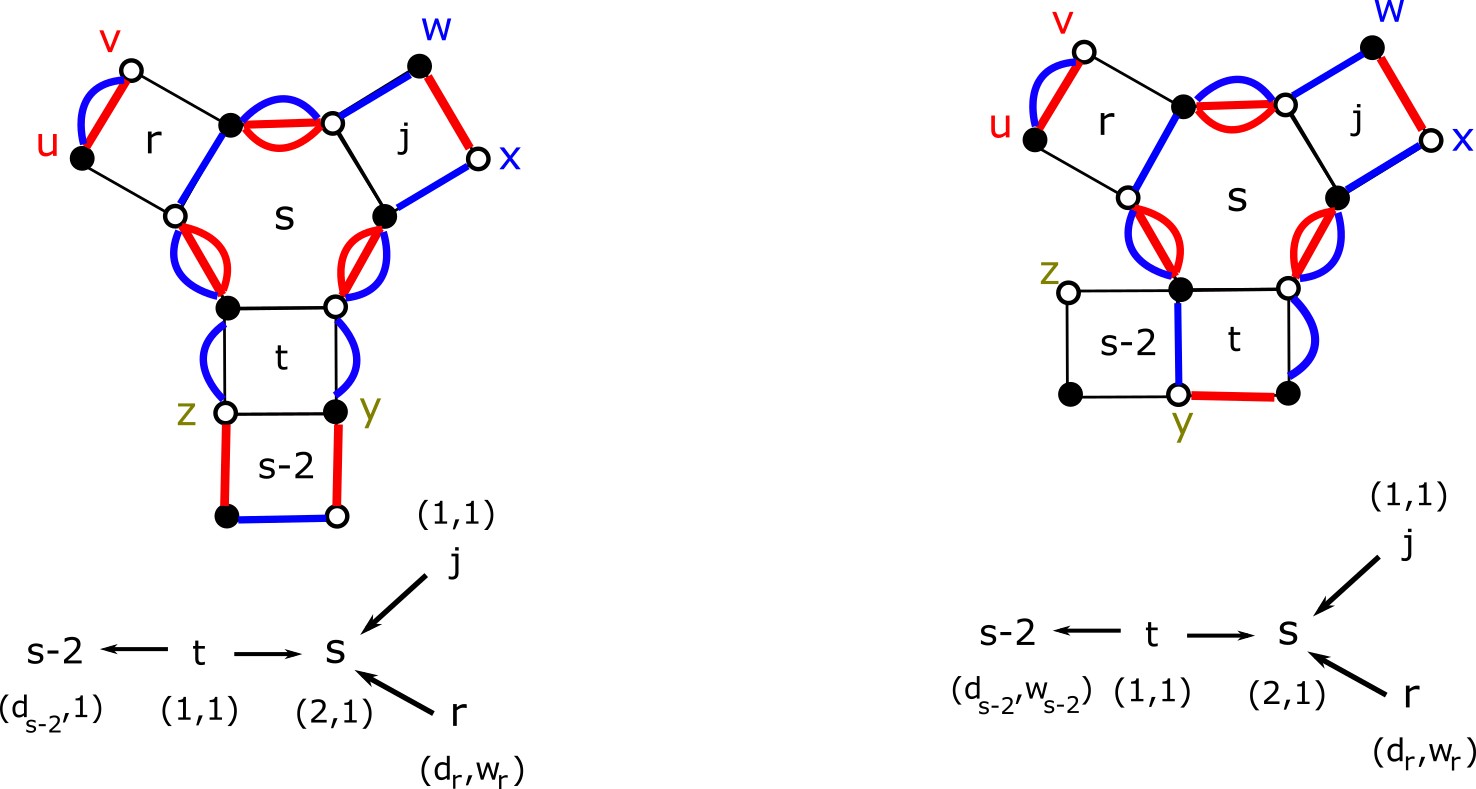}
        \caption{ }
        \label{fig:fig11}
    \end{figure}
        
        Note that in both figures, it does not matter what $(d_r,w_r)$ are because even if $w_r >0$, this does not affect the path formed between $x$ and $y$. In fact, it creates another path connecting nodes of different colors from $u$ to $z$. Therefore, we have shown that if one critical arrow was formed by enclosing $k$ tiles, then no critical arrows are created when enclosing $k+1$ tiles. Hence, by induction, we have proven that if $C_1$ encloses $n$ tiles, then $\underline{v}$ satisfies the criticality condition which establishes the base case of our induction on the number of steps of our algorithm. \allowbreak  \vspace{1em}
        
        Suppose that up the $i^{\text{th}}$ step of our algorithm, each $\underline{v}$ satisfies the criticality condition. We now aim to show that after removing $C_{i+1}$ to create $\underline{v}$, we still satisfy the criticality condition. To do this, we again induct on the number of tiles enclosed by cycle $C_{i+1}$. Suppose that $C_{i+1}$ enclosed a single tile $j$. Note that we cannot create any additional critical arrows if $j$ was enclosed in a previous cycle as $v_{_j} = 2$, so we suppose that $v_{_j} = 1$. Again, either 0 or 1 critical arrows were created with $(\underline{v}, \underline{d})$. In either case, this reduces to the argument made in the base case for $C_1$ with the only difference being that $v_{m} = 2$ for some vertices $m$. But note that this will not affect the analysis of creating critical arrows because the definition of critical arrows require that $v_{m} = 0,1$. \allowbreak  \vspace{1em}
        
        Suppose that $C_{i+1}$ encloses $k$ tiles and that the associated vector satisfies the criticality condition. Now, we aim to show that if $C_{i+1}$ encloses $k+1$ tiles, then the resulting vector $\underline{w}$ still satisfies the criticality condition. As in the base case above, note that we cannot create any additional critical arrows if $j$ was enclosed in a previous cycle as $v_{j} = 2$, so we suppose that $v_{j} = 1$. Similarly, note that even if some $v_{m} = 2$, this will not affect the analysis of creating critical arrows because the definition of critical arrows require that $v_{m} = 0,1$. So, this case reduces to the base case of the inductive step even if $w_{j-2} > 0$. \allowbreak  \vspace{1em}
        
        So, we see that by definition of the mixed dimer configuration $D$ being node-monochromatic that the criticality condition is satisfied by $\underline{v}$. By induction, enclosing $m$ tiles with $C_{i+1}$ results in $\underline{v}$ that satisfies Tran's condition (3). This gives that after each step of our algorithm, the resulting vector $\underline{v}$ must satisfy Tran's condition (3). Therefore, the resulting vector $\underline{e}$ must satisfy all of Tran's conditions. 

     \end{proof}
     
     We now show that these maps described in Theorems \ref{thm:EtoD} and \ref{thm:DtoE} are indeed inverses of one another. For an example of this, see that Examples \ref{ex:etoD} and \ref{ex:Dtoe} are inverses of each other. This completes the proof of Theorem \ref{thm:biject}.
     
     \begin{proof}
     In order to show that these maps are inverses of each other, first note that by definition of the algorithm in Theorem \ref{thm:EtoD}, we see that the number of steps in the algorithm exactly equals $|e|$, where $\underline{e}$ is the resulting vector associated to $D$. Therefore, it makes sense to induct on $|e|$.\allowbreak  \vspace{1em}
     
     For the base case, suppose that $|e| = 0$, i.e. we do not perform any steps of the algorithm in Theorem \ref{thm:EtoD}. Then, by conventions of both algorithms for Theorems \ref{thm:DtoE} and \ref{thm:EtoD}, we have that $\underline{e}$ is associated to $M_-$ and vice versa. \allowbreak  \vspace{1em}
     
     Suppose up to $k \in \nn$, when $|e| = k$, i.e. we perform $k$ steps of our algorithm in Theorem \ref{thm:EtoD}, these maps are inverses of one another. Let $\underline{e}$ such that $|e| = k+1$ where we added 1 to the $i^{\text{th}}$ entry of $\underline{e}'$ where $|e'| = |e|-1 = k$. Suppose that by Theorem \ref{thm:EtoD}, we associated the mixed dimer configuration $D'$ corresponding to $\underline{e}'$ and $D$ is the mixed dimer configuration associated to $\underline{e}$. Note that as $|e| = k+1$, we had to perform $k+1$ flips from $M_-$ to obtain $D$ and by assumption, we had to perform one more flip at tile $i$ to obtain $D$ from $D'$. \allowbreak  \vspace{1em}
     
     We aim the show that using the algorithm in Theorem \ref{thm:DtoE}, we have that our initial vector $\underline{e}$ is the vector associated to $D$. If we take the superimposition of $D \sqcup D'$, note that will consist of all 2-cycles and exactly one 4-cycle enclosing tile $i$. This implies that $D \sqcup M_-$ will also have at least one cycle containing $i$. Moreover, $D \sqcup M_-$ will have exactly one more cycle containing $i$ that that of $D' \sqcup M_-$. Hence, the corresponding vector obtained by performing the algorithm in Theorem \ref{thm:DtoE} must be $\underline{e}' + \underline{e_i}$ which is exactly $\underline{e}$. \allowbreak  \vspace{1em}
     
     Now suppose that given a mixed dimer configuration $D$, using the algorithm in Theorem \ref{thm:DtoE}, we obtain an $\underline{e}$ such that $|e|=k+1$. Let $D'$ be the mixed dimer configuration associated to $\underline{e}' = \underline{e} - \underline{e_i}$, i.e. subtracting 1 from the $i^{\text{th}}$ entry of $\underline{e}$. By the prescription of the algorithm in Theorem \ref{thm:DtoE}, we must have enclosed the tile $i$ $e_i$ number of times when looking at $D$ and $(e_i-1)$ number of 
     times when looking at $D'$. By our inductive assumption, let $\underline{e}'$ be associated to $D'$. Note that the order in which we flip tiles in the algorithm in Theorem \ref{thm:EtoD} does not matter, so to obtain the mixed dimer configuration $D$ from $D'$, it suffices to flip the tile $i$. Therefore, we see that the mixed dimer configuration associated to $\underline{e}$ prescribed by the algorithm in Theorem \ref{thm:EtoD} is exactly the mixed dimer configuration $D$ we began with.\allowbreak  \vspace{1em}
     
     Therefore, we see that these maps are indeed inverses of each other. Thus, our graphical model bijects to that of Tran's.
     \end{proof}
     \end{subsubsection}
     
     \begin{subsubsection}{Coefficients}
     To show that our model gives the exact $F$-polynomial associated to $\underline{d} \in \Phi_+$, all that is left to do is show that our model gives the right coefficients, i.e. show that the coefficient of $u_0^{e_0}\cdots u_{n-1}^{e_{n-1}}$ is $2^c$ where $c$ is the number of cycles in the mixed dimer configuration associated to $\underline{e}$. Note that when we refer to cycles, we insist the length of the cycle is larger than 2. To show our model yields the correct coefficients, we show that the number of cycles in $D$ corresponds to the number of connected components $C$ of $S$ with $\nu(C) = 0$. \allowbreak  \vspace{1em}
     
     First note that we cannot have a cycle on a tile that is a dimer configuration, i.e. we need $d_i = 2$ in order for $i$ to potentially be enclosed by a cycle. Also, note that there are no cycles in $M_-$ by the convention of the black and white coloring in the definition of $M_-$, so in order for tile $i$ to be enclosed by a cycle, we must have that $e_i \geq 1$. This tells us that the set of all tiles that can potentially be enclosed by a cycle are contained within $S = \{i \in Q~:~ (d_i,e_i) = (2,1)\}$.\allowbreak  \vspace{1em}
     
    Fix $D$ associated to $(\underline{d}, \underline{e})$. Let $C$ be a connected component of $S$. Suppose that $\nu(C) \neq 0$, that is, there is some $c \in C$ such that $c$ is the vertex of a critical arrow. Namely, this $c$ must be connected to a tile outside of $S$ by definition of the $d,e$ coordinates of the other vertex of the critical arrow. Let $c'$ be the other vertex of the critical arrow involving $c$. Note that if $c \to c'$ about the edge $\alpha$, then $(d_{c'}, e_{c'}) = (1,0)$ and we have that $n_{c,c'} = \max(d_c-d_{c'},0) + (e_{c'}-e_c) = 1-1=0$. If $c' \to c$ about the edge $\alpha$, then $(d_{c'}, e_{c'}) = (1,1)$ and we have that $n_{c',c} = \max(d_c'-d_{c},0) + (e_{c}-e_{c'}) = 0$. Therefore, by Lemma \ref{lemma:MijNij}, the edge $\alpha$ cannot have an edge from $D$ which tells us that no cycle in $D$ can enclose tile $c$. This tells us that if a cycle encloses any tile or collection of tiles in $C$, then $\nu(C) = 0$.\allowbreak  \vspace{1em}
 
    Now, suppose that $C$ is a connected component of $S$ with $\nu(C) = 0$, i.e. no critical arrows involve any vertex $c$ from $C$. Suppose that $C$ is comprised of the tiles $i, \dots, m$ where $i$ is the minimal tile in $C$ and $m$ is the maximal tile in $C$ with respect to the indexing in $Q$. We aim to show that there exists a cycle in $D$ enclosing all of $C$. Note that for $i \leq j \leq m-1$, we have that $(d_j,e_j) = (2,1)$, so there are no edges in $D$ on the edge straddling tiles $j$ and $j+1$ by Lemma \ref{lemma:MijNij} since $n_{j,j+1} = 0$. This implies that if we can show that all the boundary edges of $C$ are in $D$, that there is a cycle enclosing all of $C$ in $D$. \allowbreak  \vspace{1em}

    Note that as $e_j =1$ for all $i \leq j \leq m$, we have that each of the tiles in $C$ have been flipped exactly once from $M_-$ using our bijection of adding 1's to the $\underline{e}$-vector coinciding with flipping the corresponding tile from Theorem \ref{thm:DtoE}. Note that if a boundary edge $\alpha$ on tile $t$ is oriented black to white clockwise with respect to $G_2$, then by definition of $M_-$, $\alpha$ must have two edges distinguished in $M_-$. After one flip at tile $t$, $w(\alpha)$ decreases by 1. This gives that this edge now has exactly one edge in $D$. On the other hand, if $\alpha$ is oriented white to black clockwise with respect to $G_2$, then by definition of $M_-$, $\alpha$ has no edges distinguished in $M_-$. After performing a flip on tile $t$, we have that $w(\alpha)$ increases by 1. This gives that $\alpha$ now has exactly one edge distinguished in $D$. Hence, all the edges of the cycle enclosing $C$ that are on the boundary of $G$ have exactly one edge on them in $D$. \allowbreak  \vspace{1em}

    Now, we need to show that the cycle closes up on the interior edges of $G$, i.e. if $i > 0$ and/or $m < n-1$, we need to verify that the edges straddling tiles $i,i-1$ and $m,m+1$ have one edge distinguished in $D$. Note that as $i,m \in C$ with $\nu(C) = 0$, no arrow involving $i$ or $m$ is critical. By structure of $\underline{d} \in \Phi_+$, we have that $d_{i-1} \geq 1$ and $d_{m+1} \geq 1$. Seeking a contradiction, suppose that $n_{i, i-1} = 0$. Then, $\max(d_i-d_{i-1},0) = e_i-e_{i-1}$, i.e. $\max(2-d_{i-1},0) = 1-e_{i-1}$. Since $e_{i-1},d_{i-1} \in \{0,1,2\}$, this only occurs if either $d_{i-1} = 1$ giving that $e_{i-1} = 0$ or if $d_{i-1} = 2$ giving that $e_{i-1} = 1$. In the first case, this gives that the arrow $i \to i-1$ is critical which contradicts that $\nu(C) = 0$. In the second case, $i-1 \in C$ which contradicts the minimality of $i$. Hence, $n_{i,i-1} >0$ and by Lemma \ref{thm:DtoE}, this means there must be an edge straddling tiles $i, i-1$ in $D$. Note that this same argument holds for showing that there must be an edge in $D$ straddling tiles $m,m+1$. Therefore, we see that one cycle is formed around all of $C$ when $\nu(C) = 0$. Hence, the number of cycles in $D$ is exactly the number of connected components $C$ of $S$ with $\nu(C) = 0$. Therefore, the coefficients from our graphical model match those in Tran's.

    \begin{ex}
          \label{ex:coeff} Note that in Examples \ref{ex:etoD} and \ref{ex:Dtoe}, the mixed dimer configuration $D$ has a cycle enclosing the tile 2. From Example \ref{d6example}, we see the term in the $F$-polynomial associated to this dimer configuration $D$ has the coefficient of $2^1$.
    \end{ex} 
     \end{subsubsection}
\end{subsection}
\begin{subsection}{Direct Connection to Quiver Grassmannian}
Now that we have shown the connection between our model and that of Thao Tran, we can state the direct connection between our model and the quiver Grassmannian. 

\begin{cor}\label{thm: reptheory}
Let $Q$ be an acyclic quiver of type $D_n$ and $\underline{d}$ be a positive root in the root system. Let $M$ be the corresponding indecomposable representation of $Q$ of dimension $\underline{d}$. Then

$$\sum_{D \in P} 2^{\#\{\text{cycles in }D\}} y(D) = \sum_{\substack{\underline{e} \in \zz^n \text{ such that }\\ 0 \leq e_i \leq d_i}} \chi_{\underline{e}}(M) y_1^{e_1} \cdots y_n^{e_n}$$

\noindent where $\chi_{\underline{e}}(M)$ is the Euler Poincar\'e characteristic of the Quiver Grassmannian $Gr_{\underline{e}}(M)$.
\end{cor}

Loosely speaking, an acceptable flip at tile $i$ corresponds to increasing the dimension of the vector space at vertex $i \in Q$ in a subrepresentation $N$ of $M$. Moreover, our notion of node-connectivity corresponds to avoiding conflicts in choices of 1-dimensional subspaces of 2-dimensional vector spaces. Paths between nodes of different colors indicate inconsistent choices of subspaces on the vertices in $N$. 
\end{subsection}
\end{section}

\begin{section}{Computation of $g$-vectors from our model}
\label{sec:$g$-vectors}

Since we have established that our dimer configuration model gives a graph-theoretic interpretation of the $F$-polynomial, it is natural to want an interpretation of the $g$-vector from our model. By \cite{ca4}, the data of $F$-polynomials and $g$-vectors is equivalent to the data of Laurent expansions of cluster variables. In particular, this data is in a sense richer because it yields the Laurent expansions of the cluster variables decorated with principal coefficients.\allowbreak  \vspace{1em}

To give the interpretation of the $g$-vector, we must establish an edge weighting on the minimal matching $M_-$ associated to some acyclic quiver $Q$ of type $D_n$ and a positive root $\underline{d} \in \Phi_+$. Note that this is a separate notion from the definition of ``weighted flip" given in Section \ref{defn:weightedflip}. Following the edge weights of \cite{ms}, we define the edge weights of the base graph associated to $Q$ as follows:

    \begin{enumerate}
        \item For any internal edge $e$, define $\text{wt}(e) = 1$.
        \item For any arrow $j \to i$, on tile $j$, weight a white to black clockwise with respect to $j$ boundary edge by $x_i$. On tile $i$, weight a black to white clockwise with respect to $i$ boundary edge by $x_j$.
        
        \item Otherwise, weight the edge 1. 
    \end{enumerate}
    
    \begin{rmk}
    In step 2 of the definition of the edge weights from above, there is an ambiguity. Namely if there are multiple black to white clockwise with respect to tile $j$ boundary edges $\alpha, \beta$ of $j$, we can choose either $\alpha$ or $\beta$ to assign weight of $x_i$. It turns out that this choice does not affect the edge-weighting of matchings given the proof of Theorem \ref{thm:gvec} below.
    \end{rmk}

Define the weight of the minimal matching $M_- = M_-(Q,\underline{d})$, denoted $\text{wt}(M_-)$, to be the product of the edge weights distinguished in $M_-$. Let's see this definition in the following example:

\begin{ex}
\label{ex:g-vec}
   Let $Q$ be the $D_6$ quiver in Example \ref{ex:etoD}. Then, we assign the following weights to the edges of the base graph:
   
   \begin{center}
       \includegraphics[scale=.35]{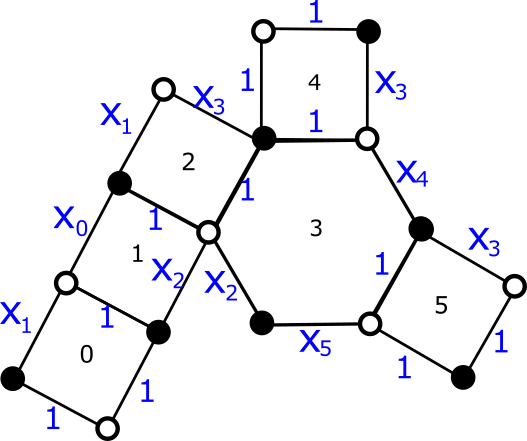}
   \end{center}
   
   If we consider the positive root $\underline{d} = (1,1,2,2,1,1)$, then we obtain the following minimal matching which has $\text{wt}(M_-) = x_1^3x_2^2x_3^2$.
   
   \begin{center}
       \includegraphics[scale=.35]{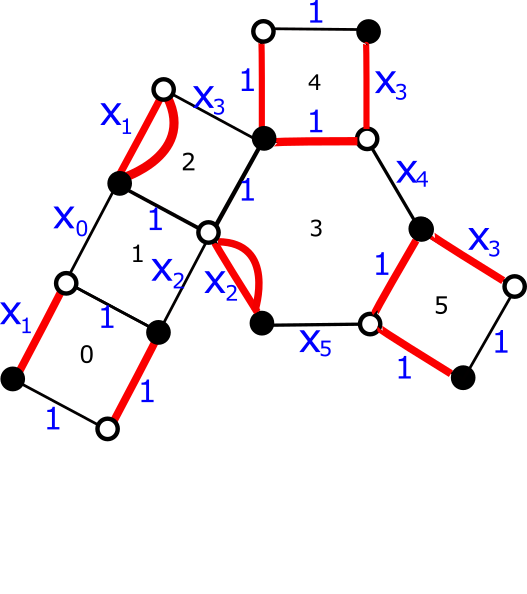}
   \end{center}
   
\end{ex}
 
With this choice of weighting, we have the following result:

\begin{thm}
\label{thm:gvec}
The $g$-vector associated to $Q$ and a positive root $\underline{d} \in \Phi_+$, denoted $\underline{g} = \underline{g}(Q, \underline{d})$, is given by 

$$\underline{g} = \deg\left(\frac{\text{wt}(M_-)}{\underline{x}^{\underline{d}}}\right)$$

\noindent where $\deg(x_0^{a_0}x_1^{a_1} \cdots x_{n-1}^{a_{n-1}}) := (a_0,a_1, \dots, a_{n-1})$.
\end{thm}

Before proving this result, we return to Example \ref{ex:g-vec}. 

\begin{ex}
      Through direct computation, the associated $g$-vector is given by \newline $(-1,2,0,0,-1,-1)$ and using Theorem \ref{thm:gvec}, we see that
      $$\deg\left(\frac{\text{wt}(M_-)}{\underline{x}^{\underline{d}}}\right) = \deg\left(\frac{x_1^3x_2^2x_3^2}{x_0x_1x_2^2x_3^2x_4x_5}\right) = \deg(x_0^{-1}x_1^2x_2^0x_3^0x_4^{-1}x_5^{-1}) = (-1,2,0,0,-1,-1).$$
      
\end{ex}

\begin{proof}
Theorem 4.4 in \cite{tran} gives an interpretation of the $g$-vector in terms of the data $(\underline{e}, \underline{d})$ where $\underline{e}$ is a vector that again satisfies conditions (1)-(3) in Theorem \ref{thm:Tran}. Namely it says that given a classical type acyclic quiver $Q$, the  $\underline{g}$-vector $\underline{g}_{\underline{d}}$ corresponding to $\underline{d} \in \Phi_+$ is

$$-\sum_{i=1}^n d_i \underline{e_i} + \sum_{i,j=1}^n d_i \max(-b_{ji}, 0) \underline{e_j}$$
where $b_{ji}$ denotes the $(j,i)^{\text{th}}$ entry of the exchange matrix associated to the quiver $Q$. Note that we can re-express this statement for the $\underline{g}-$vector in terms of the quiver $Q$ and the positive root $\underline{d} \in \Phi_+$:

$$\underline{g} = \deg\left(\frac{\displaystyle \prod_{i \in Q_0} \prod_{j \to i} x_j^{d_i}}{\underline{x}^{\underline{d}}}\right).$$

This is because the denominator of the above expression coincides with Tran's term $-\sum_{i=1}^n d_i \underline{e_i}$ and since we weight the black to white clockwise orientation with $x_i$ with $j \to i$ which coincides with Tran's term $\sum_{i,j=1}^n d_i \max(-b_{ji}, 0) \underline{e_j}$. Hence, we can re-express Tran's Theorem 4.4 in this way. Moreover, by definition of $M_-$, this expression is the same as 

$$\underline{g} = \deg\left(\frac{\text{wt}(M_-)}{\underline{x}^{\underline{d}}}\right).$$
\end{proof} 

\begin{cor}
\label{cor:laurent}
Theorem \ref{thm:main} and Theorem \ref{thm:gvec} give the full Laurent expansions of cluster variables. Namely, for an acyclic quiver $Q$ of type $D_n$, and positive root $\underline{d} \in \Phi_+$, we have that 

$$x_{Q, \underline{d}} = F_{\underline{d}}(\hat{y_0}, \hat{y_1}, \dots, \hat{y_{n-1}})x_0^{g_0}x_1^{g_1}\cdots x_{n-1}^{g_{n-1}},$$
where $\hat{y_i} = y_i \prod_{j=0}^{n-1} x_j^{-b_{ji}^0}$.
\end{cor} 

To see what the Laurent expansions for the cluster variables are, let's look at an example.
\begin{ex}
\label{ex:laurent}
Let $\underline{d} = (1,1,2,1,1) \in \Phi_+$ and let $Q$ be
\begin{center}
\begin{tikzpicture}
\node at (0,0) (0){$0$};
\node at (1,0) (1){$1$};
\node at (2,0) (2){$2$};
\node at (3,.5) (3){$3$};
\node at (3,-.5) (4){$4$};
\draw[->, thick] (1) -- (0);
\draw[->, thick] (2) -- (1);
\draw[->, thick] (3) -- (2);
\draw[->, thick] (2) -- (4);
\end{tikzpicture}
\end{center}
Then our $F$-polynomial is given by
 \begin{align*}
            F_{\underline{d}} &= 1+u_0+u_4+u_0u_1+u_0u_4+u_2u_4+u_0u_1u_4+ u_0u_1u_2+u_0u_2u_4\\
            &+2u_0u_1u_2u_4+u_0u_1u_2u_3u_4 + u_0u_1u_2^2u_4+ u_0u_1u_2^2u_3u_4
    \end{align*}
and our $g$-vector is $\underline{g} = (-1,0,0,-1,-1)$. In order to get our cluster variables, we evaluate the $F$-polynomial at $\hat{y_i}$ and then multiply by $\underline{x}^{\underline{g}}$. Note that 

$$\hat{y_0} = y_0x_1^{-1}~~~,~~~ \hat{y_1} = y_1x_0x_2^{-1}~~~,~~~\hat{y_2} = y_2x_1x_3^{-1}x_4~~~,~~~\hat{y_3} = y_3x_2~~~,~~~\hat{y_4} = y_4x_2^{-1}$$

This gives that the $F$-polynomial is

 \begin{align*}
            F_{\underline{d}}(\hat{y_0}, \hat{y_1}, \hat{y_2}, \hat{y_3}, \hat{y_4}) &= 1+y_0x_1^{-1}+y_4x_2^{-1}+\frac{y_0y_1x_0}{x_1x_2}+\frac{y_0y_4}{x_1x_2}+\frac{y_2y_4x_1x_4}{x_2x_3}+\frac{y_0y_1y_4x_0}{x_1x_2^2}\\
            &+ \frac{y_0y_1y_2x_0x_4}{x_3}+\frac{y_0y_2y_4x_4}{x_2}
            +\frac{2y_0y_1y_2y_4x_0x_4}{x_2^2x_3}+\frac{y_0y_1y_2y_3y_4x_0x_4}{x_2x_3}\\
            &+ \frac{y_0y_1y_2^2y_4x_0x_1x_4^2}{x_2^2x_3^2}+ \frac{y_0y_1y_2^2y_3y_4x_0x_1x_4^2}{x_2x_3^2}.
    \end{align*}
Therefore, the cluster variable is given by 
\begin{align*}
x_{\underline{d}} &= \frac{x_3}{x_0x_4}(1+\frac{y_0}{x_1}+\frac{y_4}{x_2}+\frac{y_0y_1x_0}{x_1x_2}+\frac{y_0y_4}{x_1x_2}+\frac{y_2y_4x_1x_4}{x_2x_3}+\frac{y_0y_1y_4x_0}{x_1x_2^2}\\
&+ \frac{y_0y_1y_2x_0x_4}{x_2x_3}+\frac{y_0y_2y_4x_4}{x_2x_3} +\frac{2y_0y_1y_2y_4x_0x_4}{x_2^2x_3}+\frac{y_0y_1y_2y_3y_4x_0x_4}{x_2x_3}\\
&+ \frac{y_0y_1y_2^2y_4x_0x_1x_4^2}{x_2^2x_3^2} + \frac{y_0y_1y_2^2y_3y_4x_0x_1x_4^2}{x_2x_3^2} ).
\end{align*}

The cluster variable's Laurent expansion can be seen term by term using the edge weighting in the poset in Figure \ref{fig:laurent}. 

\end{ex}

This example is made precise through the following formula mimicking \cite{mswpos}.
\begin{thm}[Expansion Formula]
\label{thm:expansion}
Let $D$ be the mixed dimer configuration in $P$ obtained from flipping tiles $t_1, \dots, t_\ell$ from $M_-$. Define the height monomial of $D$ by 

$$y(D) = \prod_{i=1}^\ell y_i$$
i.e. $y(D)$ is the product of all $y_i$ for which a flip occurred at tiles $i$ to obtain $D$ from $M_-$. Define $x(D) = \prod_{e \in D} w(e)$ to be the weight of the dimer configuration $D$. Then the cluster variable associated to $(Q, \underline{d})$ is given by 

$$x_{Q, \underline{d}} = \frac{1}{\underline{x}^{\underline{d}}} \sum_{D \in P} x(D) y(D).$$

\end{thm}

 \begin{figure}[H]
    \centering
        \includegraphics[scale=.22]{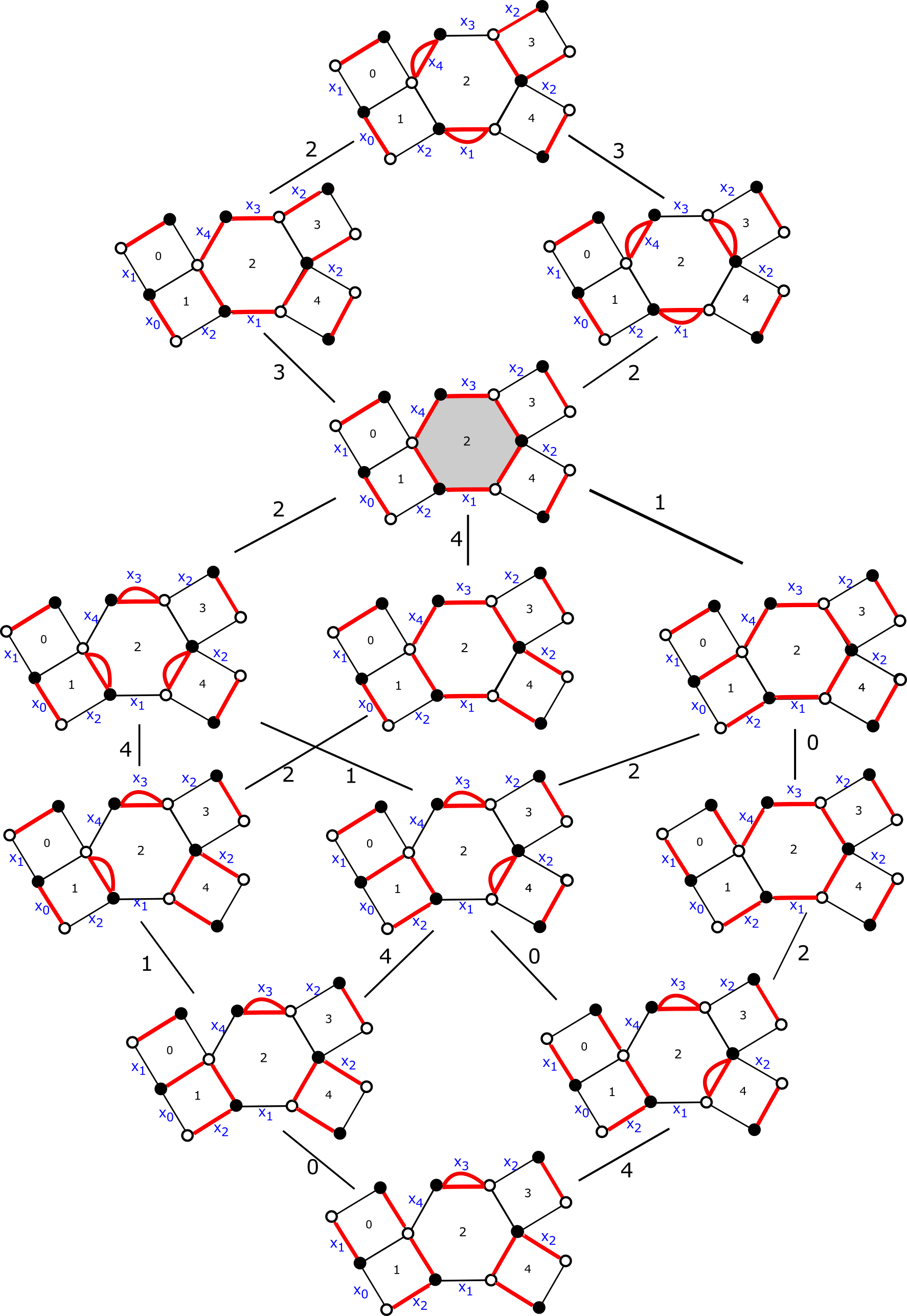}
        \caption{ }
        \label{fig:laurent}
    \end{figure}

\begin{proof}
We proceed by induction on $k = |d|$, the number of allowable flips from $M_-$ in the poset $P$. When $k=0$, our poset has one element, $M_-$. The associated $F$-polynomial is 1 and $\underline{x}^{\underline{g}} = w(M_-)$ and our expansion holds. 
Assume by induction that our claim holds up to $k \in \nn$. We aim to show that if $k+1$ flips occurred from $M_-$ i.e. $|\underline{d}| = k+1 $ to obtain the maximal mixed dimer configuration $D'$, then our expansion formula holds. It suffices to show that the term in the expansion associated to $D'$ is $\frac{1}{\underline{x}^{\underline{d}}}x(D')y(D')$. By definition of the height monomial $y(\cdot)$, we have that $y_i$'s account for flipping at each $d_i>0$ which gives that the $y_i$'s are indeed $y(D') = \underline{y}^{\underline{d}}$. Moreover, if the $k+1^{\text{st}}$ flip occurred at tile $i$, we have the weighting of the $x$'s is changed by dividing by $x_j$ where $j \to i$ in $Q$ and multiplying by $x_\ell$ when $i \to \ell$ in $Q$. Hence, the $x$ term in the term corresponding to $D'$ is indeed $x(D')$. Therefore, we have that 
 
 $$x_{Q, \underline{d}} = \frac{1}{\underline{x}^{\underline{d}}} \sum_{D \in P} x(D) y(D).$$
 \end{proof}
 
\end{section}

\begin{section}{Future Directions}
\label{sec:future}
In terms of future directions, there are a few possible avenues. For example, we would like to extend our model beyond the acyclic case. If we were to not rely on Tran's model for our proof, which requires that our quiver be acyclic, then we could potentially try to extend our dimer configuration model to apply beyond the acyclic case. Perhaps there is a proof of our model using recurrences of the $F$-polynomial or representation theory that we could utilize to use this model in less restrictive settings.\allowbreak  \vspace{1em}

Another possible direction would be to try to find a connection between our dimer configuration model and to the punctured disk model with a possible comparison to formulas for tagged arcs from \cite{mswpos}. The punctured disk model is the surface model for type $D$ cluster algebras. When working with surfaces with punctures, i.e. marked points in the interior of the surface, we can use the notion of tagged arcs in a triangulation which can then be translated into $\gamma$-symmetric matchings of an associated snake graph. We suspect that there is a way we can use a double dimer configuration model via an ``unfolding" method, which is similar to our methods from this paper, to give the connection between double dimer configurations and the surface model. In addition to this question, a paper from 2020 simplifies and unifies the approaches of \cite{mswpos} by considering ``good matchings of loop graphs" \cite{wilson2020}. These loop graphs add additional structure to the snake graphs in \cite{mswpos} in order to provide a framework to write expansion formulae for the variables corresponding to any tagged arc in a surface. With this new machinery, we wonder if using this additional structure on our dimer configurations could help find the precise connection to the punctured disk model or yield any other interesting generalizations. \allowbreak  \vspace{1em}

We also believe our interpretation is of independent combinatorial interest for the case of type $D_n$ cluster algebras, especially since these arise as cases of the recently investigated cluster structures of Schubert varieties \cite{serhiyenko2019cluster}. This new model for cluster variables has the potential to be extended to other cluster algebras - potentially shedding light on cluster structures on various Grassmannians. Namely, we believe that this mixed dimer configuration approach can improve our understanding of certain cluster variables
that arise as quadratic differences in terms of Pl\"ucker coordinates.  In particular, it is of interest to see if we could embed our base graphs into the disk to create plabic graphs, and relate them to cluster algebras from Schubert varieties.  This may dovetail with recent work in progress of Moriah Elkin \cite{Elkin-UROP}, which has been followed up by joint work in progress with the two authors \cite{EMW}. We have conjectural double and triple dimer interpretations for Laurent expansions associated to cluster variables in cluster algebras associated to Grassmannians of types $Gr(3,6)$, $Gr(3,7)$, and $Gr(3,8)$.

\end{section}

\newpage

\bibliographystyle{alpha}
\bibliography{bibliography}

\end{document}